\documentclass[11pt,a4paper]{article} \usepackage[top=2.5cm,bottom=2.5cm,left=2.5cm,right=2.5cm]{geometry}
\usepackage{graphicx}
\usepackage{epstopdf}
\usepackage{stmaryrd}
\usepackage[usenames, dvipsnames]{color}
\usepackage{amsmath,amssymb,mathrsfs,amsthm}
\usepackage{cancel}
\usepackage{enumerate}
\usepackage{subfigure}
\usepackage{float}
\usepackage{multirow}
\usepackage{booktabs}
\usepackage{color}
\usepackage{arydshln}

\usepackage[colorlinks,citecolor=red,linkcolor=blue]{hyperref}

\newtheorem{theorem}{Theorem}
\newtheorem{lemma}{Lemma}
\newtheorem{remark}{Remark}
\newtheorem{example}{Example}

\newtheorem{definition}{Definition}

\newtheorem{assumption}{Assumption}

\numberwithin{equation}{section}
\numberwithin{theorem}{section}
\numberwithin{lemma}{section}
\numberwithin{example}{section}
\numberwithin{definition}{section}

\title{Two time-stepping schemes for sub-diffusion equations with singular source terms}
\author{Han Zhou\thanks{Institute of Mathematics, Hebei University of Technology, Tianjin 300401, China. This author was partially supported by the National Natural Science Foundation of China (No. 11901151). Email: zhouhan@hebut.edu.cn}
        \and
        Wenyi Tian\thanks{Corresponding author. Center for Applied Mathematics, Tianjin University, Tianjin 300072, China. This author was partially supported by the National Natural Science Foundation of China (Nos. 11701416 and 12071343). Email: twymath@gmail.com}
}
\date{}

\begin{document}
\maketitle
\begin{abstract}
  Singular source terms in sub-diffusion equations may lead to the unboundedness of solutions, which will bring a severe reduction of convergence order of existing time-stepping schemes. In this work, we propose two efficient time-stepping schemes for solving sub-diffusion equations with a class of source terms mildly singular in time. One discretization is based on the Gr{\"u}nwald-Letnikov and backward Euler methods. First-order error estimate with respect to time is rigorously established for singular source terms and nonsmooth initial data. The other scheme derived from the second-order backward differentiation formula (BDF) is proved to possess second-order accuracy in time. Further, piecewise linear finite element and lumped mass finite element discretizations in space are applied and analyzed rigorously. Numerical investigations confirm our theoretical results.

  {\bf Keywords:} sub-diffusion equation, singular source term, convolution quadrature, backward differentiation formula, linear finite element, lumped mass finite element

  {\bf AMS subject classifications:} 65M06, 65M60, 65M15, 35R11, 35R05
\end{abstract}

\section{Introduction}
This paper concerns with the construction of efficient discrete schemes for the sub-diffusion equation with a singular source term in time, that is
\begin{equation}\label{eq:tfdesub0}
  {^C}D^{\alpha}_tu(x,t)-\Delta u(x,t)=f(x,t), ~~~ (x,t)\in \Omega\times (0, T],
\end{equation}
together with the Dirichlet boundary condition and the nonsmooth initial condition
\begin{equation}\label{eq:ibconditions}
  \begin{aligned}
    &u(x,t)=0,~~&&x\in\partial\Omega,~t>0,\\
    &u(x,0)=u^{0}(x),~~&&x\in \Omega,
  \end{aligned}
\end{equation}
where $\Omega\subset\mathbb{R}^d$, $d=1,2$, and $u^{0}(x)$ belongs to $L^{2}(\Omega)$.
The notation ${^C}D^{\alpha}_tu$ denotes the $\alpha$-th ($0<\alpha<1$) order left Caputo derivative of $u$ with respect to variable $t$, which is defined by
\begin{equation*}
  {^C}D^{\alpha}_tu(t)=\frac{1}{\Gamma(1-\alpha)}\int_0^t(t-\tau)^{-\alpha}u'(\tau)\mathrm{d}\tau,
\end{equation*}
where $\Gamma(\cdot)$ represents the Gamma function given by $\displaystyle\Gamma(s)=\int_0^{\infty}t^{s-1}e^{-t}\mathrm{d}t$ with $\Re(s)>0$.
Throughout this paper, we will restrict our consideration to the singular source term $f(x,t)\in L^{1}(0,T;L^{2}(\Omega))$.

Time-fractional diffusion equations ($0<\alpha<1$) were formulated in \cite{SchneiderW1989} and used to simulate anomalous diffusion phenomena in physics recently \cite{Metzler2000}. In contrast to some regularity results for classical second-order parabolic problems, extensive analyses have shown that the solution of a time-fractional evolution problem usually exhibits a weakly singular property near the origin even if the given data are sufficiently smooth with respect to time \cite{MillerA:1971,Brunner1986,SakamotoY:2011,Stynes2017}.

In terms of numerical approximations to this type of problems, most high-order time discretization methods were originally proposed by assuming that the solutions are relatively regular for temporal variable. For instance, the so-called L1 and L1-2, etc, schemes based on continuous piecewise polynomial interpolation were separately proposed and analyzed theoretically in \cite{Lin2007,GaoSZ2014,LvX2016}. Furthermore, for sub-diffusion equations with nonsmooth initial values, applications of the time-stepping schemes from above may lead to the order reduction to first order in time \cite{JinLZ:2016}, and the optimal convergence order can be preserved by a correction approach in \cite{Yan2018}. However, the convergence order deteriorates significantly near the initial layer \cite{Stynes2017}. Then the L1 type schemes on graded meshes named after \cite{Brunner1985} were designed and analyzed rigorously to improve the order of accuracy \cite{Stynes2017,LiaoLZ2018,Kopteva2020}.

On the other hand, convolution quadrature (CQ) based on linear multistep methods was proposed and analyzed in the pioneering work \cite{Lubich1986}. In particular, first- and second-order fractional backward differentiation formulae were used as time discretizations of fractional diffusion-wave equations \cite{LubichST1996,Cuesta2006}. To restore the order of convergence, a strategy with the help of discrete Laplace transform was proposed in \cite{LubichST1996} by choosing proper weight coefficients with respect to the source term and initial values in the discrete schemes. Furthermore, this approach was applied in \cite{JinLZ2017} to develop proper corrected schemes based on fractional $k$-step BDFs for approximating both sub-diffusion and fractional diffusion-wave equations in time. It is proved that the $k$th-order convergence rate can be achieved if the source terms possess sufficient regularity in time.

An alternative approach by correction was applied in \cite{Cuesta2006} to overcome the order reduction. This idea of construction may originate from \cite{Lubich1988,Lubich2004}, which interpreted convolutions with non-integrable kernels as equivalent Hadmard-finite integrals \cite{ELLIOTT1993}. It was also used in \cite{JLZ2016}, where CQ methods based on backward Euler (BE) and second-order backward difference (SBD) were revisited and developed for time discretizations of both sub-diffusion and diffusion-wave equations. In addition, the numerical results in \cite[Table 5]{JLZ2016} showed a second-order scheme in solving the sub-diffusion equation with certain continuous source term in time, while the mechanism behind was unknown yet.

To the best of our knowledge, most existing time-stepping schemes for equation \eqref{eq:tfdesub0} are limited to the source term possessing certain degree of smoothness at time $t=0$. For instance, error estimates for corrected BE and SBD schemes in \cite{JLZ2016,JinLZ2017} were established under the conditions $\|f(0)\|<\infty, \int_0^t(t-s)^{\alpha-1}\|f'(s)\|\mathrm{d}s<\infty$ and $\|f(0)\|<\infty, \|f'(0)\|<\infty, \int_0^t(t-s)^{\alpha-1}\|f''(s)\|\mathrm{d}s<\infty$, respectively. For corrected high-order schemes in \cite{JinLZ2017}, additional regularity conditions on source terms were required to restore high-order accuracy. However, those smoothness conditions are not suitable for singular source terms considered in this paper, such as $f(x,t)=t^\mu g(x)$ with $-1<\mu<0$. Thus the convergence orders of the existing time-stepping schemes are completely lost and far below one.

As indicated above, singular source terms will make the problem much more difficult and challenging both in designing efficient time-stepping schemes and establishing error estimates.
In this paper, based on the previous works \cite{JLZ2016,Cuesta2006,Lubich1988,Lubich2004}, we dedicate to designing efficient numerical schemes to solve \eqref{eq:tfdesub0}-\eqref{eq:ibconditions} with the singular source terms $f(x,t)\in L^{1}(0,T;L^{2}(\Omega))$ satisfying Assumption \ref{af} as well as nonsmooth initial values. First, we investigate the existence and uniqueness of weak solutions of the problem in the space $C((0,T]; L^{2}(\Omega))$.
Furthermore, by the works on FEM and lumped mass FEM in \cite{ChatzipantelidisLT:2012,JinLZ:2013,JinLPZ:2015,LubichST1996,Thomee:2006}, we applied these two spatial discretization methods for the problem \eqref{eq:tfdesub0}-\eqref{eq:ibconditions}. After providing a possible understanding on the efficiency of the approach in \cite{Lubich1988,Lubich2004}, we propose two new time-stepping schemes \eqref{eq:GLBE} and \eqref{eq:FBDF22}, named by GLBE and FBDF22 schemes, respectively.

Our main contribution consists of the following aspects.
\begin{itemize}
  \item[-] The well-posedness and regularity of solutions of \eqref{eq:tfdesub0}-\eqref{eq:ibconditions} with the singular source term $f(x,t)$ satisfying Assumption \ref{af} are investigated. Error estimates of semidiscrete continuous piecewise linear FEM and lumped mass FEM are rigorously established.
  \item[-] For $f(x,t)$ satisfying Assumption \ref{af}, two new time-stepping schemes, named by GLBE \eqref{eq:GLBE} and FBDF22 \eqref{eq:FBDF22}, are proposed. Error estimates for the first-order and second-order accurate schemes are rigorously established, respectively.
      This largely improves the convergence order of the previous time-stepping schemes in solving \eqref{eq:tfdesub0}-\eqref{eq:ibconditions} with singular source terms.
  \item[-] Due to the singularity of the source term in time, the discrete Laplace transform technique commonly used in existing works can not be employed for analyzing the fully discrete schemes proposed in this paper. Then we develop a new analysis technique based on the Laplace transform of the source term rather than its generating function to estimate the errors of the proposed schemes \eqref{eq:GLBE} and \eqref{eq:FBDF22}.
\end{itemize}

The rest of this paper is organized as follows. In Section \ref{sec:2}, we study the existence and uniqueness of the weak solutions in certain proper space for the problem \eqref{eq:tfdesub0}-\eqref{eq:ibconditions} with singular source terms  with respect to time.
In Section \ref{sec:3}, the continuous piecewise linear FEM and lumped mass FEM are used for spatial discretization. Semidiscrete error estimates of both methods are established. In Section \ref{sec:4}, two new time-stepping schemes based on first- and second-order CQ-BDFs were proposed for temporal discretization. Error estimates for fully discrete solutions are rigorously established. Section \ref{sec:5} presents several numerical examples to verify the theoretical convergence rates in both spatial and temporal directions as estimated in Sections \ref{sec:3} and \ref{sec:4}.

\section{Well-posedness and regularity of solutions} \label{sec:2}
As discussed in \cite{Bajlekova:2001,SakamotoY:2011}, the well-posedness and regularity of problem \eqref{eq:tfdesub0}-\eqref{eq:ibconditions} have been well established for $f(x,t)\in L^{p}(0,T;L^{2}(\Omega))$ with $p>1$.
In this section, we will revisit the problem and investigate the existence, uniqueness and regularity of its solution with a singular source term satisfying Assumption \ref{af}.

\begin{assumption}\label{af}
  The singular source term $f(x,t)$ in \eqref{eq:tfdesub0} is assumed to be in $L^1(0,T;L^{2}(\Omega))$ such that its Laplace transform with respect to time $t$ is analytic within the domain $\Sigma_{\theta}$ given by \eqref{eq:sigt} and satisfies $\|\hat{f}(s)\|\le c|s|^{-\mu-1}$ for $-1<\mu<0$.
\end{assumption}

For instance, the singular source term $f(x,t)=t^{\mu}g(x)$ with $-1<\mu<0$ and $g(x)\in L^2(\Omega)$ satisfies the conditions in Assumption \ref{af}.

The Laplace operator $\Delta$ is symmetric, then satisfies the following resolvent estimate \cite{Thomee:2006}
\begin{equation}\label{eq:resolv}
  \|\big(s-\Delta\big)^{-1}\|\le M|s|^{-1},\quad \forall~s\in\Sigma_{\theta}
\end{equation}
for $\theta\in (\pi/2,\pi)$, where $\Sigma_{\theta}$ is a sector of the
complex plane $\mathbb{C}$ given by
\begin{equation}\label{eq:sigt}
  \Sigma_{\theta}=\big\{z\in\mathbb{C}\setminus\{0\}:|\mathrm{arg}z|<\theta\big\}.
\end{equation}

First, we define a weak solution of the problem \eqref{eq:tfdesub0}-\eqref{eq:ibconditions} analogous to that in \cite[Chapter 7.1.1]{Evans:2010}. Throughout this paper, the notation $(\cdot,\cdot)$ denotes the inner product in $L^{2}(\Omega)$.
\begin{definition}\label{def:ws}
  For $0<\alpha<1$, we call a function $u\in L^{1}(0, T; H_{0}^{1}(\Omega))$ with ${^C}D^{\alpha}_tu\in L^{1}(0, T; H^{-1}(\Omega))$ a weak solution of \eqref{eq:tfdesub0}-\eqref{eq:ibconditions} provided that
  \begin{equation}\label{eq:vf}
    \begin{aligned}
      \langle{^C}D^{\alpha}_tu,\varphi\rangle+(\nabla u,\nabla\varphi)&=(f,\varphi),\quad\forall~\varphi\in H_{0}^{1}(\Omega),\\
      u(x,0)&=u^0(x)
    \end{aligned}
  \end{equation}
  holds for a.e. $t\in (0,T]$.
\end{definition}

A weak solution defined by Definition \ref{def:ws} satisfies \eqref{eq:tfdesub0} for a.e. $t\in (0,T]$. Note that the operator $-\Delta$ is symmetric, and its eigenvalues and the corresponding eigenfunctions on the domain $\Omega$ with a homogeneous Dirichlet boundary condition are denoted by $\{\lambda_{k}\}_{k=1}^{+\infty}$ and $\{\varphi_{k}\}_{k=1}^{+\infty}$, where $0<\lambda_{1}\leq \lambda_{2}\leq\cdots$, and $-\Delta\varphi_{k}=\lambda_{k}\varphi_{k}$ in $\Omega$ and $\varphi_{k}=0$ on $\partial\Omega$. Then $\{\varphi_{k}\}_{k=1}^{+\infty}$ makes up an orthonormal basis in $L^{2}(\Omega)$.
If $u(x,t)$ satisfies Definition \ref{def:ws}, then by substituting
$u(x,t)=\sum_{k=1}^{+\infty}u_{k}(t)\varphi_{k}(x)$ into the variational form \eqref{eq:vf} and taking $\varphi=\varphi_{k}, k=1,2,\cdots$, we may obtain ${^C}D_t^{\alpha}u_{k}(t)+\lambda_{k}u_{k}(t)=(f, \varphi_{k})$. Multiplying it by $\varphi_{k}(x)$ and summing over all $k=1,2,\cdots$ yields the equation \eqref{eq:tfdesub0} for a.e. $t\in (0,T]$.

The existence, uniqueness and regularity of the solution to the problem \eqref{eq:tfdesub0}-\eqref{eq:ibconditions} are stated in the following theorems.
\begin{theorem}\label{thm:wpr}
  Let $u^{0}(x)\equiv0$ and $f(x,t)$ in \eqref{eq:tfdesub0} satisfy Assumption \ref{af}. Then the problem \eqref{eq:tfdesub0}-\eqref{eq:ibconditions} has a unique solution $u\in C((0, T]; L^{2}(\Omega))$, which satisfies
  \begin{equation}\label{eq:2.3}
    \|u(t)\|\leq ct^{\alpha+\mu}, \quad -1<\mu<0,~~t>0.
  \end{equation}
\end{theorem}
\begin{proof}
  Taking Laplace transform on \eqref{eq:tfdesub0} with $u^{0}(x)\equiv0$ arrives at $\hat{u}(x,s)=(s^{\alpha}-\Delta)^{-1}\hat{f}(x,s)$.
  Let $\tilde{u}(x,t)$ denote the inverse Laplace transform of $\hat{u}(x,s)$, then it follows that
  \begin{equation}\label{eq:tildeu}
    \tilde{u}(x,t)=\frac{1}{2\pi i}\int_{\Gamma}e^{s t}(s^{\alpha}-\Delta)^{-1}\hat{f}(x,s)\mathrm{d}s,
  \end{equation}
  where
  \begin{equation}\label{gammac}
    \Gamma=\{\sigma+iy:\sigma>0,~ y\in\mathbb{R}\}.
  \end{equation}
  By the resolvent estimate in \eqref{eq:resolv} and the Cauchy's theorem, $\Gamma$ in \eqref{eq:tildeu} can be replaced by $\Gamma_{\varepsilon}^{\theta}\cup S_{\varepsilon}$, where $\theta\in (\pi/2,\pi)$ and
  \begin{equation}\label{eq:gammavts}
    \Gamma_{\varepsilon}^{\theta}\cup S_{\varepsilon}=\{\rho e^{\pm i\theta}:~ \rho\geq\varepsilon\}\cup\{\varepsilon e^{ix}: -\theta\leq x\leq \theta\}.
  \end{equation}

  Due to the condition in Assumption \ref{af}, it holds that $\|\hat{f}(x,s)\|\le c|s|^{-\mu-1}$ with $-1<\mu<0$.
  Let $\varepsilon=t^{-1}$ in \eqref{eq:gammavts}, then an estimate for $\tilde{u}(x,t)$ in \eqref{eq:tildeu} can be obtained as follows
  \begin{equation}\label{eq:tildeunorm}
    \begin{aligned}
      \|\tilde{u}(x,t)\|&\leq c\int_{\Gamma_{\varepsilon}^{\theta}\cup S_{\varepsilon}}|e^{st}|\|(s^{\alpha}-\Delta)^{-1}\|\|\hat{f}(x,s)\||\mathrm{d}s| \\
        & \leq c\int_{\Gamma_{\varepsilon}^{\theta}\cup S_{\varepsilon}}|e^{st}||s|^{-\alpha-\mu-1}|\mathrm{d}s|   \\
        & \leq c\left(\int_{\varepsilon}^{+\infty}e^{\rho t\cos\theta}\rho^{-\alpha-\mu-1}\mathrm{d}\rho+\varepsilon^{-\alpha-\mu}
                  \int_{-\theta}^{\theta}e^{\varepsilon t\cos\xi}\mathrm{d}\xi\right) \\
        & \leq ct^{\alpha+\mu}.
    \end{aligned}
  \end{equation}

  Next we prove that $\tilde{u}$ is continuous with respect to $t\in(0,T]$. For any $t_1,t_2>0$, using \eqref{eq:tildeu}, we have
  \begin{equation}\label{eq:conut}
    \begin{split}
      \|\tilde{u}(x,t_1)-\tilde{u}(x,t_2)\|&\le c\Big\|\int_{\Gamma_{\varepsilon}^{\theta}\cup S_{\varepsilon}}\left(e^{st_1}-e^{s t_2}\right)(s^{\alpha}-\Delta)^{-1}\hat{f}(x,s)\mathrm{d}s\Big\|  \\
        &\leq c\int_{\Gamma_{\varepsilon}^{\theta}\cup S_{\varepsilon}}|e^{st_1}-e^{st_2}||s|^{-\alpha-\mu-1}|\mathrm{d}s|\\
        &\leq c|t_1-t_2|\max\{t_1^{\alpha+\mu-1},t_2^{\alpha+\mu-1}\},
    \end{split}
  \end{equation}
  where the last inequality holds by the result
  \begin{equation*}
    |e^{st_1}-e^{s t_2}|=\Big|\int_{t_2}^{t_1}se^{sz}\mathrm{d}z\Big|\leq\Big|
    \int_{t_2}^{t_1}|s||e^{sz}|\mathrm{d}z\Big|\leq m|s||t_1-t_2|
  \end{equation*}
  with $m=\max\{|e^{st_1}|,|e^{st_2}|\}$. The estimate \eqref{eq:conut} implies the continuity of $\tilde{u}(x,t)$ with respect to $t>0$.

  Recall that any weak solution $u$ defined by \eqref{eq:vf} satisfies \eqref{eq:tfdesub0} and \eqref{eq:ibconditions} for a.e. $t\in (0,T]$, then it has the same Laplace transform as that of $\tilde{u}$, which implies $u=\tilde{u}$ for a.e. $t\in (0,T]$. Therefore, the problem \eqref{eq:tfdesub0}-\eqref{eq:ibconditions} has a unique weak solution in the space $C((0,T]; L^{2}(\Omega))$, and the result \eqref{eq:2.3} follows from \eqref{eq:tildeunorm}.
\end{proof}

For the case $u^{0}\in L^{2}(\Omega)$ and $f(x,t)\equiv0$, the corresponding result can be obtained analogous to that in \cite{SakamotoY:2011}.
\begin{theorem}[\cite{SakamotoY:2011}]
  Let $u^{0}\in L^{2}(\Omega)$ and $f\equiv0$ in the problem \eqref{eq:tfdesub0}-\eqref{eq:ibconditions}. Then there exists a unique weak solution $u\in C([0, T]; L^{2}(\Omega))$ such that
  \begin{equation}
    \max_{0\leq t\leq T}\|u(t)\|\leq c\|u^{0}\|.
  \end{equation}
\end{theorem}

\section{Spatially semidiscrete FEM}\label{sec:3}
In this section, we establish error estimates for the semidiscrete Galerkin FEM and lumped mass FEM for the sub-diffusion equation \eqref{eq:tfdesub0}-\eqref{eq:ibconditions} with a singular source term $f(x,t)$ satisfying Assumption \ref{af}.
\subsection{Galerkin FEM}\label{subsec:3.1}
Let $\mathcal{T}_h$ be a regular triangulation of $\Omega$ into $d$-simplexes and $h=\max\limits_{T\in\mathcal{T}_h}\mathrm{diam}(T)$ the maximal diameter, then we
denote $X_{h}\subset H_0^1(\Omega)$ as a continuous piecewise linear finite element space on $\mathcal{T}_h$. The semidiscrete problem by finite element for \eqref{eq:tfdesub0}-\eqref{eq:ibconditions} is to find $u_{h}(t)\in X_{h}$ satisfying
\begin{equation}\label{eq:sFE}
  \begin{aligned}
    ({^C}D^{\alpha}_tu_{h}(t),\varphi)+(\nabla u_{h}(t),\nabla\varphi)&=(f(t),\varphi),\quad \forall~ \varphi\in X_{h},\\
    u_h(0)&=P_hu^0,
  \end{aligned}
\end{equation}
where the operator $P_{h}:L^{2}(\Omega)\to X_{h}$ denotes the $L^{2}$-projection onto the finite element space $X_{h}$, defined by
$$(P_{h}\varphi, \psi)=(\varphi, \psi), ~\forall~\psi\in X_{h}.$$
We further introduce the operator $\Delta_{h}: X_{h}\rightarrow X_{h}$ defined by \begin{equation}\label{eq:Deltah}
  -(\Delta_{h}\varphi, \psi)=(\nabla\varphi,\nabla\psi),~\forall~\varphi, \psi\in X_{h}.
\end{equation}
Then the semidiscrete form of \eqref{eq:sFE} can be rewritten in the form of
\begin{equation}\label{eq:sFEO}
  {^C}D^{\alpha}_tu_{h}(t)-\Delta_{h} u_{h}(t)=f_{h}(t),~\forall~t>0,
\end{equation}
with $u_{h}(0)=P_{h}u^{0}$ and $f_{h}=P_{h}f$.
The Laplace transform on \eqref{eq:sFEO} implies
\begin{equation}\label{eq:LTsFEO}
  s^{\alpha}\hat{u}_{h}(s)-\Delta_{h}\hat{u}_{h}(s)=s^{\alpha-1}P_{h}u^{0}+\hat{f}_{h}(s).
\end{equation}
Then by the inverse Laplace transform together with the estimate $\|(s^{\alpha} -\Delta_{h})^{-1}\|\leq M|s|^{-\alpha}$ for $s\in\Sigma_{\theta}$ with $\theta\in(\pi/2,\pi)$ \cite{Thomee:2006}, the solution $u_{h}(t)$ for $t>0$ can be represented by
\begin{equation}\label{eq:uh}
  u_{h}(t)=\frac{1}{2\pi i}\int_{\Gamma_{\varepsilon}^{\theta}\cup S_{\varepsilon}}
  e^{st}(s^{\alpha}-\Delta_{h})^{-1}\left(s^{\alpha-1}P_{h}u^{0}
  +\hat{f}_{h}(s)\right)\mathrm{d}s.
\end{equation}

Next we establish the error estimate of the semidiscrete scheme \eqref{eq:sFE} with homogeneous initial data and the singular source term $f$ satisfying Assumption \ref{af}.
\begin{theorem}\label{thm:errsFE}
  Assume that $u^0(x)\equiv0$ and $f(x,t)$ in \eqref{eq:tfdesub0} satisfies Assumption \ref{af}. Let $u$ and $u_h$ be the solutions to \eqref{eq:vf} and \eqref{eq:sFE}, respectively. Then we have
  \begin{equation}\label{eq:errsFEM0}
    \|u(t)-u_h(t)\|\le ct^{\mu}h^2,\quad -1<\mu<0,~~ t>0.
  \end{equation}
\end{theorem}

The proof of Theorem \ref{thm:errsFE} is presented in Appendix \ref{sec:app1}.
For the case $u^0(x)\in L^2(\Omega)$ and $f(x,t)\equiv0$, the corresponding result estimated in \cite{JinLZ:2013,JinLZ:2019} is $\|u(t)-u_h(t)\|\le ch^2|\log h|t^{-\alpha}\|u^0\|$, which can be improved to be $\|u(t)-u_h(t)\|\le ch^2t^{-\alpha}\|u^0\|$ removing $|\log h|$ as mentioned in \cite{JinLZ:2019}. The improved result can be obtained by using the similar argument in the proof of Theorem \ref{thm:errsFE}.
\begin{theorem}\label{thm:errsFE1}
  Assume $u^0(x)\in L^2(\Omega)$ and $f(x,t)\equiv0$ in \eqref{eq:tfdesub0}-\eqref{eq:ibconditions}. Let $u$ and $u_h$ be the solutions to \eqref{eq:vf} and \eqref{eq:sFE}, respectively. Then we have
  \begin{equation}\label{eq:errsFEM2}
    \|u(t)-u_h(t)\|\le ct^{-\alpha}\|u^0\|h^2,\quad t>0.
  \end{equation}
\end{theorem}

\subsection{Lumped mass FEM}
In this subsection, we consider the more practical lumped mass
FEM \cite{Thomee:2006} and estimate the corresponding discretization errors.
The semidiscrete problem by the lumped mass FEM for \eqref{eq:tfdesub0}-\eqref{eq:ibconditions} is to find $\bar{u}_{h}(t)\in X_{h}$ satisfying
\begin{equation}\label{eq:slmFE}
  \begin{aligned}
    ({^C}D^{\alpha}_t\bar{u}_{h}(t),\varphi)_h+(\nabla\bar{u}_{h}(t),\nabla\varphi)&=(f(t),\varphi),\quad \forall~ \varphi\in X_{h},\\
    \bar{u}_h(0)&=P_hu^0,
  \end{aligned}
\end{equation}
where $(\cdot,\cdot)_h$ is defined by
$$
  (v,w)_h:=\sum_{\tau\in\mathcal{T}_h}Q_{\tau,h}(vw) \;\quad\hbox{with}\quad \; Q_{\tau,h}(g)=\frac{|\tau|}{d+1}\sum_{j=1}^{d+1}g(x_j^\tau),
$$
and
$\{x_j^\tau\}_{j=1}^{d+1}$ are the vertices of the $d$-simplex $\tau\in\mathcal{T}_h$. Define the operator $\bar{\Delta}_{h}: X_{h}\rightarrow X_{h}$ by
\begin{equation}\label{eq:bDeltah}
  -(\bar{\Delta}_{h}\varphi, \psi)_h=(\nabla\varphi,\nabla\psi),~\forall~\varphi, \psi\in X_{h},
\end{equation}
and a projection operator $\bar{P}_h:L^2(\Omega)\rightarrow X_h$ by
$$(\bar{P}_{h}f, \psi)_h=(f,\psi),~\forall~\psi\in X_{h}.$$
Then the semidiscrete scheme \eqref{eq:slmFE} can be rewritten in the form of
\begin{equation}\label{eq:slmFEO}
  {^C}D^{\alpha}_t\bar{u}_{h}(t)-\bar{\Delta}_{h} \bar{u}_{h}(t)=\bar{f}_{h}(t),~\forall~t>0,
\end{equation}
with $\bar{u}_{h}(0)=P_{h}u^{0}$ and $\bar{f}_{h}=\bar{P}_{h}f$.

\begin{theorem}\label{thm:errslmFE}
  Assume that $u^0(x)\equiv0$ and $f(x,t)$ in \eqref{eq:tfdesub0} satisfies Assumption \ref{af}. Let $u(t)$ and $\bar{u}_h(t)$ be the solutions to \eqref{eq:vf} and \eqref{eq:slmFE}, respectively. Then we have
  \begin{equation}\label{eq:errslmFEM0}
    \|u(t)-\bar{u}_h(t)\|\le ct^{\mu}h,\quad -1<\mu<0,~~t>0.
  \end{equation}
  Moreover, if the quadrature error operator $Q_h$ defined by \eqref{eq:Qh} satisfies \eqref{eq:seQh}, then we have
  \begin{equation}\label{eq:errslmFEM1}
    \|u(t)-\bar{u}_h(t)\|\le ct^{\mu}h^2,\quad -1<\mu<0,~~t>0.
  \end{equation}
\end{theorem}

The proof of Theorem \ref{thm:errslmFE} is given in Appendix \ref{sec:app2}.
For $u^0(x)\in L^2(\Omega)$ and $f(x,t)\equiv0$, the error estimate for the lumped mass finite element scheme \eqref{eq:slmFE} has been considered in \cite{JinLZ:2013}. In the following theorem, we provide an improved result without the term $|\log h|$ by using the similar argument in the proof of Theorem \ref{thm:errslmFE}.
\begin{theorem}\label{thm:errslmFE1}
  Assume $u^0(x)\in L^2(\Omega)$ and $f(x,t)\equiv0$. Let $u(t)$ and $\bar{u}_h(t)$ be the solutions to \eqref{eq:sFE} and \eqref{eq:slmFE}, respectively. Then we have the error $u(t)-\bar{u}_h(t)$ satisfies
  \begin{equation}\label{eq:errslmFEM3}
    \|u(t)-\bar{u}_h(t)\|\le ct^{-\alpha}\|u^0\|h,\quad t>0.
  \end{equation}
  Moreover, if the quadrature error operator $Q_h$ defined by \eqref{eq:Qh} satisfies \eqref{eq:seQh}, then we have
  \begin{equation}\label{eq:errslmFEM4}
    \|u(t)-\bar{u}_h(t)\|\le ct^{-\alpha}\|u^0\|h^2,\quad t>0.
  \end{equation}
\end{theorem}

\section{Time discretization} \label{sec:4}
In this section, we construct two fully discrete schemes for solving \eqref{eq:tfdesub0}-\eqref{eq:ibconditions} with the singular source term $f(x,t)$ satisfying Assumption \ref{af} and establish the error estimates in time. Without loss of generality, our discussion is mainly on the semidiscrete scheme \eqref{eq:sFEO}. Analogous results for the lump mass FEM can be obtained by the same technique.

If the singular source term satisfies Assumption \ref{af}, then the result of Theorem \ref{thm:wpr} reveals that the analytic solution to the problem \eqref{eq:tfdesub0}-\eqref{eq:ibconditions} will be unbounded near the origin as well for $-1<\mu<-\alpha$. Singularity of the source term and the solution will bring a severe influence on the accuracy of numerical results. We present an example to illustrate the phenomenon of order reduction of the schemes based on backward Euler (BE) and second-order BDF (SBD) in \cite{JLZ2016,JinLZ2017}.
\begin{example}\label{ex:8.1}
  Consider the fractional ordinary differential equation (fODE) ${^C}D^{\alpha}u(t)=\lambda u(t)+f(t)$ for $t\in (0,T]$, with $u^{0}=0$ and $f(t)=\frac{\Gamma(\nu+1)}{\Gamma(\nu+1-\alpha)} t^{\nu-\alpha}-\lambda t^{\nu}$, where $\alpha-1<\nu<0$ and $\lambda=-1$. The exact solution is $u(t)=t^{\nu}$.
\end{example}

The existing corrected BE and uncorrected SBD schemes (see \cite{JLZ2016,JinLZ2017}) for the fODE in Example \ref{ex:8.1} are given by
\begin{equation} \label{eq:4.8}
  \tau^{-\alpha}\sum_{j=0}^{n}\sigma_{j}\left(u_{n-j}-u^{0}\right)=\lambda u_{n}+f(t_{n})
\end{equation}
for $1\leq n \leq N$, where $\tau=T/N$ and $\sigma_{j}$ being the coefficients of $(1-\xi)^\alpha$ or $(\frac{3}{2}-2\xi+\frac{1}{2}\xi^{2})^\alpha$. Note that the term $f(t)$ in Example \ref{ex:8.1} is unbounded at $t=0$, then the corrected SBD scheme in \cite{JLZ2016,JinLZ2017} is not applicable to Example \ref{ex:8.1}. Thus the order of accuracy of the uncorrected SBD scheme for the fODE in Example \ref{ex:8.1} can not exceed one.

The errors $e_{N}$ defined by $|u_{N}-u(T)|$ with various $N$ such that $N\tau=T$ are presented in Tables \ref{ta:2} and \ref{ta:2b}. Convergence rates are checked by the formula $rate=\log_{2}\left(e_{N}/{e_{2N}}\right)$ with an average. As shown in Tables \ref{ta:2} and \ref{ta:2b}, the corrected BE and uncorrected SBD schemes in \eqref{eq:4.8} both fail to restore first order of accuracy for Example \ref{ex:8.1}.

\begin{table}[ht!]
  \centering
  \caption{Errors and convergence rates for Example \ref{ex:8.1} at $T=1$ by the corrected BE scheme \eqref{eq:4.8}.}
  \label{ta:2} 
  \setlength{\tabcolsep}{4pt}
  \begin{tabular}{|c|c|ccccc|c|}
    \hline
    $\alpha$&$\nu$ & $N=20$ &40& 80&160 &320  & rate \\
    \hline
    0.1 & -0.1 & 2.7636E-03 & 1.5279E-03 & 8.4788E-04 & 4.7180E-04 & 2.6310E-04 & $\approx$ 0.85\\
        & -0.5 & 2.0762E-02 & 1.5103E-02 & 1.1052E-02 & 8.1204E-03 & 5.9843E-03 & $\approx$ 0.45\\
        & -0.9 & 4.1489E-01 & 4.0265E-01 & 3.9146E-01 & 3.8112E-01 & 3.7153E-01 & $\approx$ 0.04\\\hline
    0.5 & -0.1 & 4.6742E-02 & 3.3543E-02 & 2.4446E-02 & 1.8019E-02 & 1.3391E-02 & $\approx$ 0.45\\
        & -0.3 & 1.1258E-01 & 9.2906E-02 & 7.7893E-02 & 6.6046E-02 & 5.6447E-02 & $\approx$ 0.25\\
        & -0.5 & 2.8959E-01 & 2.7507E-01 & 2.6512E-01 & 2.5824E-01 & 2.5345E-01 & $\approx$ 0.05\\\hline
    0.7 & -0.1 & 1.3185E-01 & 1.1051E-01 & 9.3940E-02 & 8.0587E-02 & 6.9526E-02 & $\approx$ 0.23\\
        & -0.2 & 1.9793E-01 & 1.7798E-01 & 1.6231E-01 & 1.4934E-01 & 1.3817E-01 & $\approx$ 0.13\\
        & -0.3 & 2.9962E-01 & 2.8907E-01 & 2.8275E-01 & 2.7895E-01 & 2.7666E-01 & $\approx$ 0.03\\\hline
  \end{tabular}
\end{table}
\begin{table}[ht!]
  \centering
  \caption{Errors and convergence rates for Example \ref{ex:8.1} at $T=1$ by the uncorrected SBD scheme \eqref{eq:4.8}.}	
  \label{ta:2b} 
  \setlength{\tabcolsep}{4pt}
  \begin{tabular}{|c|c|ccccc|c|}
    \hline
    $\alpha$&$\nu$ & $N=20$ &40& 80&160 &320  & rate \\
    \hline
    0.1 & -0.1 & 2.5725E-03 & 1.4344E-03 & 8.0166E-04 & 4.4881E-04 & 2.5163E-04 & $\approx$ 0.84\\
        & -0.5 & 2.0097E-02 & 1.4781E-02 & 1.0893E-02 & 8.0416E-03 & 5.9451E-03 & $\approx$ 0.44\\
        & -0.9 & 4.1427E-01 & 4.0236E-01 & 3.9131E-01 & 3.8105E-01 & 3.7149E-01 & $\approx$ 0.04\\\hline
    0.5 & -0.1 & 4.4559E-02 & 3.2453E-02 & 2.3900E-02 & 1.7746E-02 & 1.3254E-02 & $\approx$ 0.44\\
        & -0.3 & 1.0972E-01 & 9.1482E-02 & 7.7180E-02 & 6.5688E-02 & 5.6267E-02 & $\approx$ 0.24\\
        & -0.5 & 2.8695E-01 & 2.7376E-01 & 2.6447E-01 & 2.5791E-01 & 2.5328E-01 & $\approx$ 0.05\\\hline
    0.7 & -0.1 & 1.2845E-01 & 1.0873E-01 & 9.3022E-02 & 8.0116E-02 & 6.9286E-02 & $\approx$ 0.22\\
        & -0.2 & 1.9449E-01 & 1.7619E-01 & 1.6139E-01 & 1.4887E-01 & 1.3792E-01 & $\approx$ 0.12\\
        & -0.3 & 2.9664E-01 & 2.8754E-01 & 2.8198E-01 & 2.7856E-01 & 2.7646E-01 & $\approx$ 0.03\\\hline
  \end{tabular}
\end{table}

To tackle the above disadvantages of the existing time-stepping schemes, we intend to consider a reformulation of the semidiscrete scheme \eqref{eq:sFEO}, and propose two new fully discrete schemes in the next subsections to preserve the optimal first and second order of accuracy.

\subsection{GLBE scheme}
First, we define $U_{h}(t)$ and $F_{h}(t)$ by $\int_{0}^{t}u_{h}(\xi)\mathrm{d}\xi$ and $\int_{0}^{t}f_{h}(\xi)\mathrm{d}\xi$, respectively.  Then for integrable $u_{h}$ and $f_{h}$ it follows that
\begin{equation}\label{eq:u1}
  D{U}_{h}(t)={u}_{h}(t),\quad \quad {U}_{h}(0)=0
\end{equation}
and
\begin{equation}\label{eq:f1}
  DF_{h}(t)=f_{h}(t),\quad \quad  F_{h}(0)=0
\end{equation}
for a.e. $t>0$, where $D:=\partial /\partial t$. Next, we substitute \eqref{eq:u1} and \eqref{eq:f1} into \eqref{eq:sFEO} and get
\begin{equation}\label{eq:DU}
  {^C}D^{\alpha}_tDU_{h}(t)=\Delta_{h}DU_{h}(t)+DF_{h}(t).
\end{equation}
Further, integrating \eqref{eq:DU} from $0$ to $t$ and using the semigroup property of fractional integrals arrive at
\begin{equation}\label{eq:daU}
  {^C}D^{\alpha}_tU_{h}(t)=\Delta_{h}U_{h}(t)+F_{h}(t)+\frac{t^{1-\alpha}}{\Gamma(2-\alpha)}u_{h}(0),\quad
  U_{h}(0)=0.
\end{equation}
Together with \eqref{eq:u1}-\eqref{eq:f1}, the semidiscrete scheme \eqref{eq:daU} can be viewed as an equivalent form of \eqref{eq:sFEO}.

Given a uniform partition of the interval $[0, T]$ by
\begin{equation*}
  0=t_{0}<t_{1}<\cdots<t_{N-1}<t_{N}=T.
\end{equation*}
The step size of the uniform mesh is denoted by $\tau=T/N$ and $t_{n}=n\tau$ for $0\le n\le N$.

We introduce $\tilde{U}_{h}(t)$ and $\tilde{u}_{h}(t)$ as approximations to $U_{h}(t)$ and $u_{h}(t)$ solving \eqref{eq:daU} and \eqref{eq:u1}, respectively. For $t>0$,
$\tilde{U}_{h}(t)$ satisfies the difference equation
\begin{equation}\label{eq:7.3}
  D_{\tau}^{\alpha}\tilde{U}_{h}(t)=\Delta_{h}\tilde{U}_{h}(t)+F_{h}(t)
  +\frac{t^{1-\alpha}}{\Gamma(2-\alpha)}u_{h}(0),
\end{equation}
and is prescribed by zero otherwise, where $F_h(t)$ satisfying \eqref{eq:f1}. Here $D_{\tau}^{\alpha}$ denotes the well-known Gr\"unwald-Letnikov (GL) or fractional backward Euler difference operator, which is written as
\begin{equation}\label{eq:cGL}
  D_{\tau}^{\alpha}\tilde{U}_{h}(t)=\tau^{-\alpha}\sum_{j=0}^{\infty}\sigma_{j}\tilde{U}_{h}(t-j\tau),
\end{equation}
where $\{\sigma_{j},~j\ge 0\}$ are coefficients of a generating function such that $\sum\limits_{j=0}^{\infty}\sigma_{j}\xi^{j}=(1-\xi)^{\alpha}$. Moreover, we denote  $D_{\tau}$ by the backward Euler (BE) operator such that $D_{\tau}\tilde{U}_{h}(t)=\tau^{-1}(\tilde{U}_{h}(t)-\tilde{U}_{h}(t-\tau))$, and let
\begin{equation}\label{eq:7.4}
  \tilde{u}_{h}(t)=D_{\tau}\tilde{U}_{h}(t).
\end{equation}
This implies $\tilde{u}_{h}(t)=0$ for $t\leq 0$ as well.
Then choosing $t=t_{n}$ for $n=1,\cdots, N$ in \eqref{eq:7.3} and \eqref{eq:7.4}, we propose a fully discrete scheme, named by GLBE and of the form
\begin{equation}\label{eq:GLBE}
  \begin{split}
    &\tau^{-\alpha}\sum_{j=0}^{n}\sigma_{j}\tilde{U}_{h}^{n-j}=\Delta_{h}\tilde{U}_{h}^{n}+F_{h}^{n}
     +\frac{t_{n}^{1-\alpha}}{\Gamma(2-\alpha)}u_{h}(0),\\
    &\tilde{u}_{h}^{n}=\tau^{-1}\big(\tilde{U}_{h}^{n}-\tilde{U}_{h}^{n-1}\big),
  \end{split}
\end{equation}
where $\tilde{u}_{h}^{n}:=\tilde{u}_{h}(t_{n})$, $\tilde{U}_{h}^{n}:=\tilde{U}_{h}(t_{n})$
and $F_{h}^{n}:=F_{h}(t_{n})$ with $F_h(\cdot)$ satisfying \eqref{eq:f1}.

We illustrate the superiority of the above method by Example \ref{ex:8.1}, where the fODE is now discretized by
\begin{equation}\label{eq:4.8a}
  \tau^{-\alpha}\sum_{j=0}^{n}\sigma_{j}U_{n-j}=\lambda U_{n}+F(t_{n})+\frac{t_{n}^{1-\alpha}}{\Gamma(2-\alpha)}u^{0}\quad\text{and} \quad
  u_{n}=\tau^{-1}\left(U_{n}-U_{n-1}\right).
\end{equation}
Here $F(t)=\int_{0}^{t}f(\xi)\mathrm{d}\xi$.
From Table \ref{ta:3}, the first-order accuracy of the scheme \eqref{eq:4.8a} can be observed for Example \ref{ex:8.1} with various $\alpha\in (0,1)$ and $\nu\in [\alpha-1,0)$, where the convergence rates are checked by the formula $rate=\log_{2}\left(e_{N}/{e_{2N}}\right)$ with an average. In contrast, as shown in Table \ref{ta:2}, the corrected BE scheme \eqref{eq:4.8} fails to restore the first-order accuracy when the source term is singular near the origin.

\begin{table}[ht!]
  \centering
  \caption{Errors and convergence rates for Example \ref{ex:8.1} at $T=1$ by the GLBE scheme in \eqref{eq:4.8a}.}\label{ta:3} 
  \setlength{\tabcolsep}{4pt}
  \begin{tabular}{|c|c|ccccc|c|}
    \hline
    $\alpha$ & $\nu$ & $N=20$ & 40 & 80 & 160 & 320 & rate \\\hline
    0.1 & -0.1 & 2.3849E-03 & 1.1760E-03 & 5.8379E-04 & 2.9082E-04 & 1.4513E-04 & $\approx$ 1.01\\
        & -0.5 & 1.2491E-02 & 6.1250E-03 & 3.0285E-03 & 1.5042E-03 & 7.4894E-04 & $\approx$ 1.01\\
        & -0.9 & 3.3167E-02 & 1.6092E-02 & 7.8995E-03 & 3.9006E-03 & 1.9320E-03 & $\approx$ 1.03\\\hline
    0.5 & -0.1 & 1.0049E-03 & 4.0389E-04 & 1.6766E-04 & 7.1469E-05 & 3.1188E-05 & $\approx$ 1.25\\
        & -0.3 & 6.8081E-03 & 3.1971E-03 & 1.5226E-03 & 7.3162E-04 & 3.5376E-04 & $\approx$ 1.07\\
        & -0.5 & 1.6924E-02 & 8.2275E-03 & 4.0464E-03 & 2.0030E-03 & 9.9517E-04 & $\approx$ 1.02\\\hline
    0.7 & -0.1 & 8.2958E-04 & 2.1682E-04 & 2.8495E-05 & 1.9119E-05 & 2.3754E-05 & $\approx$ 1.28\\
        & -0.2 & 4.7144E-03 & 2.1506E-03 & 9.9154E-04 & 4.5987E-04 & 2.1400E-04 & $\approx$ 1.12\\
        & -0.3 & 1.0252E-02 & 5.0163E-03 & 2.4778E-03 & 1.2303E-03 & 6.1272E-04 & $\approx$ 1.02\\\hline
  \end{tabular}
\end{table}

We next establish the fully discrete error estimates by means of Laplace transform.
By taking Laplace transform on \eqref{eq:u1},\eqref{eq:f1} and \eqref{eq:daU}, the semidiscrete solution $u_{h}(t)$ in \eqref{eq:uh} can be rewritten as
\begin{equation}\label{eq:uh1}
  u_{h}(t)=\frac{1}{2\pi i}\int_{\Gamma_{ \varepsilon}^{\theta}\cup S_{\varepsilon}}e^{st}s\hat{U}_{h}(s)\mathrm{d}s,
\end{equation}
where $\Gamma_{ \varepsilon}^{\theta}\cup S_{\varepsilon}$ is defined by \eqref{eq:gammavts} and
\begin{equation}\label{eq:hatUh}
  \hat{U}_{h}(s)=(s^{\alpha}-\Delta_{h})^{-1}
  \big(\hat{F}_{h}(s)+s^{\alpha-2}u_{h}(0)\big).
\end{equation}
In addition, in view of the definition of $F_{h}$ satisfying \eqref{eq:f1} and Assumption \ref{af}, we get
\begin{equation}\label{eq:Fs}
  \|\hat{F}_{h}(s)\|=|s|^{-1}\|\hat{f_{h}}(s)\|= |s|^{-1}\|P_{h}\hat{f}(s)\|\le |s|^{-1}\|\hat{f}(s)\|\le c|s|^{-\mu-2}.
\end{equation}

Since $\tilde{U}_{h}(t)$ and $\tilde{u}_{h}(t)$ are solutions of \eqref{eq:7.3} and \eqref{eq:7.4}, respectively, we have the representations $\tilde{U}_{h}(x,t)=\sum\limits_{j=1}^{M-1}\tilde{U}_{j}^{h}(t)\varphi_{j}^{h}(x)$ and $\tilde{u}_{h}(x,t)=\sum\limits_{j=1}^{M-1}\tilde{u}_{j}^{h}(t)\varphi_{j}^{h}(x)$, where $\varphi_{j}^{h}$ is the eigenfunction corresponding to the eigenvalue $\lambda_{j}^{h}$ of the operator $-\Delta_{h}$ such that $-(\Delta_{h}\varphi_{j}^{h},\psi)=\lambda_{j}^{h}(\varphi_{j}^{h},\psi)$ for any $\psi\in X_h$ and each $1\leq j\leq M-1$. Then substituting expressions of $\tilde{u}_{h}$ and $\tilde{U}_{h}$ into \eqref{eq:7.3} and \eqref{eq:7.4} yields that
\begin{equation}
  \tilde{u}_{j}^{h}(t)=\sum_{k=0}^{\infty}p_{k}H(t-t_{k})\big(F_{j}(t-t_{k})+v_{j}(t-t_{k})\big),\quad  1\leq j\leq M-1,~~t>0,
\end{equation}
where $F_{j}(t)=(F_{h}(t),\varphi_{j}^{h})$, $v_{j}(t)=v(t)(u^{0},\varphi_{j}^{h})$ with $v(t)=t^{1-\alpha}/\Gamma(2-\alpha)$, and $H(t)$ denotes the Heaviside function which equals to one for $t\geq 0$ and zero otherwise. The coefficients $\{p_{k}\}_{k=0}^{\infty}$ are generated by the power series of $\omega(z)\big(\omega(z)^{\alpha}+\lambda_{j}^{h}\big)^{-1}$ with the notation $\omega(z):=\tau^{-1}(1-z)$.  From the fact that $F_{h}(t)$ and $v(t)$ are both continuous functions in time and zero at $t=0$, it follows that $\tilde{u}_{j}^{h}(t)$ is continuous  for each $j$ when $t>0$. Therefore, we can rewrite $\tilde{u}_{h}$ in \eqref{eq:7.4} as
\begin{equation}\label{eq:buh}
  \tilde{u}_{h}(t)=\frac{1}{2\pi i}\int_{\Gamma}e^{st}\omega(e^{-s\tau})\widehat{\tilde{U}_{h}}(s)\mathrm{d}s,
\end{equation}
where $\Gamma$ is given by \eqref{gammac}. Moreover, with $\tilde{U}_{h}(t)=0$ for $t\leq 0$, we have from \eqref{eq:cGL} that
\begin{equation*}
  \widehat{D_{\tau}^{\alpha}\tilde{U}_{h}}(s)=\tau^{-\alpha}
  \sum_{j=0}^{\infty}\sigma_{j}\int_{0}^{+\infty}e^{-s t}\tilde{U}_{h}(t-j\tau)\mathrm{d}t
  =\omega(e^{-s\tau})^{\alpha}\widehat{\tilde{U}_{h}}(s),
\end{equation*}
and then it yields from \eqref{eq:7.3} that
\begin{equation}\label{eq:bUh}
  \widehat{\tilde{U}_{h}}(s)=\big(\omega(e^{-s\tau})^{\alpha}-\Delta_{h}\big)^{-1}
  \big(\hat{F}_{h}(s)+s^{\alpha-2}u_{h}(0)\big).
\end{equation}

\begin{theorem}\label{le:81}
  Assume that $u^0(x)\equiv0$ and $f(x,t)$ in \eqref{eq:tfdesub0} satisfies Assumption \ref{af}. Let $u_{h}$ and $\tilde{u}_{h}^n$ be the solutions to \eqref{eq:sFEO} and \eqref{eq:GLBE}, respectively. Then we have
  \begin{equation}\label{eq:81}
    \|u_{h}(t_{n})-\tilde{u}_{h}^{n}\|\leq c\left(t_{n}^{\alpha+\mu-1}\tau +t_{n}^{\alpha-2}\tau^{2+\mu}\right),~~-1<\mu<0,~~1\le n\le N.
  \end{equation}
\end{theorem}
\begin{proof}
  From \eqref{eq:7.3}, \eqref{eq:7.4}, \eqref{eq:GLBE} and \eqref{eq:buh}, we can represent the solution $\tilde{u}_{h}^{n}$ to \eqref{eq:GLBE} as
  \begin{equation}\label{eq:bbuh}
    \tilde{u}_{h}^{n}=\tilde{u}_{h}(t_n)=\lim_{L\to+\infty}\frac{1}{2\pi i}\int_{\sigma-iL}^{\sigma+iL}
    e^{st_n}\omega(e^{-s\tau})\widehat{\tilde{U}_{h}}(s)\mathrm{d}s.
  \end{equation}
  Then for any $L>0$ and fixed $\tau>0$, there exists $\bar{N}\in\mathbb{N}^{+}$ such that $\left( 2\bar{N}+1\right)\pi/\tau\le L\le (2\bar{N}+3)\pi/\tau$, and the integral in \eqref{eq:bbuh} can be divided into three parts
  \begin{equation}\label{eq:iut}\small
    \begin{split}
    \int_{\sigma-iL}^{\sigma+iL}e^{st_n}\omega(e^{-s\tau})\widehat{\tilde{U}_{h}}(s)\mathrm{d}s
         &=\Big(\int_{\sigma+i(2\bar{N}+1)\frac{\pi}{\tau}}^{\sigma+iL}
         +\int_{\sigma-i(2\bar{N}+1)\frac{\pi}{\tau}}^{\sigma+i(2\bar{N}+1)\frac{\pi}{\tau}}+\int_{\sigma-iL} ^{\sigma-i(2\bar{N}+1)\frac{\pi}{\tau}}\Big)
         e^{st_n}\omega(e^{-s\tau})\widehat{\tilde{U}_{h}}(s)\mathrm{d}s.
    \end{split}
  \end{equation}
  For the first integral in \eqref{eq:iut}, it follows from \eqref{eq:bUh}, \eqref{eq:Fs}, Lemmas \ref{le:g0} and \ref{le:45} that
  \begin{equation*}
    \begin{split}
      \|\int_{\sigma+i(2\bar{N}+1)\pi/\tau}^{\sigma+iL}e^{st_n}\omega(e^{-s\tau})
      \widehat{\tilde{U}_{h}}(s)\mathrm{d}s\|
       &\leq c\int_{\sigma+i(2\bar{N}+1)\pi/\tau}^{ \sigma+i(2\bar{N}+3)\pi/\tau}|e^{st_n}||\omega(e^{-s\tau})|^{1-\alpha}
       |s|^{-\mu-2}|\mathrm{d}s|   \\
       &\leq c\int_{(2\bar{N}+1)\pi/\tau}^{(2\bar{N}+3)\pi/\tau}e^{\sigma t_n}(\sigma+y-2\bar{N}\pi/\tau)^{1-\alpha}y^{-\mu-2}\mathrm{d}y \\
       &\leq c\frac{\tau^{\alpha+\mu}}{(2\bar{N}+1)^{\mu+2}}e^{\sigma t_n},
    \end{split}
  \end{equation*}
  where $\sigma=t_n^{-1}$ is chosen. Then the above bound tends to zero when $L\to+\infty$ ($\bar{N}\to+\infty$). Analogous result holds for the third integral in \eqref{eq:iut} as well.

  Next we consider the estimate of the second integral in \eqref{eq:iut}. First, some integral curves are introduced as follows:
  \begin{equation}
    \Gamma_{\varepsilon,\tau}^{\theta}=\{\rho e^{\pm i\theta}:~ \varepsilon\leq\rho\leq \pi/(\tau\sin\theta)\},
  \end{equation}
  \begin{equation}
    \Gamma^{+}=\{x+i(2\bar{N}+1)\pi/\tau:~ \pi/\tau\cot\theta\le x\le \sigma\},
  \end{equation}
  \begin{equation}
    \Gamma^{-}=\{x-i(2\bar{N}+1)\pi/\tau:~ \pi/\tau\cot\theta\le x\le \sigma\}.
  \end{equation}
  As shown in \eqref{eq:bUh}, $\widehat{\tilde{U}_{h}}(s)$ is analytic in the sector $\Sigma_{\theta}$. Using the Cauchy's theorem and the periodic property of exponential function, we obtain
  \begin{equation}\label{eq:812}
    \begin{split}
      &\int_{\sigma-i(2\bar{N}+1)\pi/\tau}^{\sigma+i(2\bar{N}+1)\pi/\tau}
      e^{s t_n}\omega(e^{-s\tau})\widehat{\tilde{U}_{h}}(s)\mathrm{d}s \\
      =&\int_{\Gamma^{-}\cup \Gamma^{+}}e^{s t_n}\omega(e^{-s\tau})\widehat{\tilde{U}_{h}}(s)\mathrm{d}s+\int_{\Gamma_{\varepsilon,\tau}^{\theta}\cup S_{\varepsilon}}e^{s t_n}\omega(e^{-s\tau})\widehat{\tilde{U}_{h}}(s)\mathrm{d}s   \\
      +&\sum_{\substack{p=-\bar{N} \\ p\ne 0}}^{\bar{N}}\int_{\Gamma_{0,\tau}^{\theta}}e^{(s+i2p\pi/\tau) t_n}\omega(e^{-s\tau})\left(\omega(e^{-s\tau})^{\alpha}-\Delta_{h}\right)^{-1}\hat{F}_{h}(s+i2p\pi/\tau)\mathrm{d}s.   \\
    \end{split}
  \end{equation}
  Moreover, by taking $\sigma=t_n^{-1}$, we have
  \begin{equation*}
    \begin{split}
      \|\int_{\Gamma^{-}\cup\Gamma^{+}}e^{s t_n}\omega(e^{-s\tau})\widehat{\tilde{U}_{h}}(s)\mathrm{d}s\|
      &\leq c\Big(\int_{\Gamma^{-}}|e^{st_n}||\omega(e^{-(s+i2\bar{N}\pi/\tau)\tau})|^{1-\alpha}
      |s|^{-\mu-2}|\mathrm{d}s|\\
      &~~~~~+\int_{\Gamma^{+}}|e^{st_n}||\omega(e^{-(s-i2\bar{N}\pi/\tau)\tau})|^{1-\alpha}|s|^{-\mu-2}|\mathrm{d}s| \Big)\\
      &\leq c\int_{\pi/\tau\cot\theta}^{\sigma}e^{xt_n}(|x|+\pi/\tau)^{1-\alpha}
      |(2\bar{N}+1)\pi/\tau|^{-\mu-2}\mathrm{d}x  \\
      &\leq c\frac{\tau^{\alpha+\mu}}{(2\bar{N}+1)^{\mu+2}}e^{\sigma t_n},
    \end{split}
  \end{equation*}
  which tends to zero for $L\to+\infty$ ($\bar{N}\to+\infty$).
  Hence, from \eqref{eq:uh1}, \eqref{eq:bbuh} and \eqref{eq:812}, it follows that
  \begin{equation}\label{eq:gerror}
    \begin{split}
      u_{h}(t_n)-\tilde{u}_{h}^n&=\int_{\Gamma_{\varepsilon}^{\theta}\backslash\Gamma_{\varepsilon,\tau}^{\theta}}e^{st_n}
      s\hat{U}_{h}(s)\mathrm{d}s+\int_{\Gamma_{\varepsilon,\tau}^{\theta}\cup S_{\varepsilon}}e^{s t_n}\left(s\hat{U}_{h}(s)-\omega(e^{-s\tau})\widehat{\tilde{U}_{h}}(s)\right)\mathrm{d}s\\
      &-\sum_{\substack{p=-\infty \\ p\ne 0}}^{+\infty}\int_{\Gamma_{0,\tau}^{\theta}}e^{(s+i2\pi p/\tau) t_n}\omega(e^{-s\tau})(\omega(e^{-s\tau})^{\alpha}-\Delta_{h})^{-1}
                       \hat{F}_{h}(s+i2\pi p/\tau)\mathrm{d}s \\
                       &:=I_{1}+I_{2}+I_{3}.
    \end{split}
  \end{equation}
  The estimation of the first item $I_{1}$ in \eqref{eq:gerror} is given by
  \begin{equation*}
    \begin{split}
      \|I_{1}\|&\leq c \int_{\Gamma_{\varepsilon}^{\theta}\backslash\Gamma_{\varepsilon,\tau}^{\theta}}
      |e^{st_n}||s|^{-\alpha-\mu-1}|\mathrm{d}s|
         \leq c\int_{\frac{\pi}{\tau\sin\theta}}^{+\infty}
      e^{\rho t_n\cos\theta}\rho^{-\alpha-\mu-1}\mathrm{d}\rho\\
         &\leq c\tau\int_{\frac{\pi}{\tau\sin\theta}}^{+\infty}
      e^{\rho t_n\cos\theta}\rho^{-\alpha-\mu}\mathrm{d}\rho
          \leq c t_n^{\alpha+\mu-1}\tau.
    \end{split}
  \end{equation*}
  To estimate $I_{2}$ in \eqref{eq:gerror}, we rewrite it as the summation of two parts, i.e., $I_{2}=I_{2}'+I_{2}''$, where
  \begin{equation*}
    I_{2}'=\int_{\Gamma_{\varepsilon,\tau}^{\theta}\cup S_{\varepsilon}}e^{s t_n}\left(s-\omega(e^{-s\tau})\right)\hat{U}_{h}(s)\mathrm{d}s
      ~~~\text{and}~~~
    I_{2}''=\int_{\Gamma_{\varepsilon,\tau}^{\theta}\cup S_{\varepsilon}}e^{s t_n}\omega(e^{-s\tau})\left(\hat{U}_{h}(s)-\widehat{\tilde{U}_{h}}(s)\right)\mathrm{d}s.
  \end{equation*}
  Let $\varepsilon=t_n^{-1}$, then the result in Lemma \ref{le:45} implies that
  \begin{equation}
    \begin{split}
      \|I_{2}'\|&\leq c\tau\int_{\Gamma_{\varepsilon, \tau}^{\theta}\cup S_{\varepsilon}}|e^{st_n}||s|^{-\alpha-\mu}|\mathrm{d}s| \\
        &\leq c\tau\left(\int_{\varepsilon}^{\frac{\pi}{\tau\sin\theta}}e^{\rho t_n \cos\theta}\rho^{-\alpha-\mu}\mathrm{d}\rho+\int_{-\theta}^{\theta}e^{ \varepsilon t_n\cos\xi}\varepsilon^{-\alpha-\mu+1}\mathrm{d}\xi\right) \\
        &\leq c t_n^{\alpha+\mu-1}\tau.
    \end{split}
  \end{equation}
  On the other hand, using Lemmas \ref{le:g0} and \ref{le:45} arrives at
  \begin{equation*}
    \begin{split}
      &\|(s^{\alpha}-\Delta_{h})^{-1}
      -\left(\omega(e^{-s\tau})^{\alpha}-\Delta_{h}\right)^{-1}\|\\
      \leq& \|\left(\omega(e^{-s\tau})^{\alpha}-\Delta_{h}\right)^{-1}\|
      \|\omega(e^{-s\tau})^{\alpha}-s^{\alpha}\| \|\left(s^{\alpha}-\Delta_{h}\right)^{-1}\| \\
      \leq &\tau|s||\omega(e^{-s\tau})|^{-\alpha}
    \end{split}
  \end{equation*}
  for $s$ enclosed by curves $\Gamma_{0,\tau}^{\theta}$, $\Im(s)=\pm \pi/\tau$ and $\Gamma$.
  This yields
  \begin{equation}
    \begin{split}
      \|I_{2}''\|&\leq c\tau\int_{\Gamma_{\varepsilon,\tau}^{\theta}\cup S_{\varepsilon}}|e^{s t_n}||\omega(e^{-s\tau})|^{1-\alpha}|s|^{-\mu-1}|\mathrm{d}s|  \\
          &\leq \tau\left(\int_{\varepsilon}^{\frac{\pi}{\tau\sin\theta}}e^{\rho t_n\cos\theta }\rho^{-\alpha-\mu}\mathrm{d}\rho+\int_{-\theta}^{\theta}e^{ \varepsilon t_n\cos\xi}\varepsilon^{-\alpha-\mu+1}\mathrm{d}\xi\right) \\
      &\leq c t_n^{\alpha+\mu-1}\tau,
    \end{split}
  \end{equation}
  where $\varepsilon=t_n^{-1}$ is taken. In addition, for any $\mu>-1$, from the inequality
  \begin{equation*}
    \sum_{p=1}^{+\infty}p^{-\mu-2}\leq 1+\int_{1}^{+\infty}p^{-\mu-2}\mathrm{d}p\leq 1+\frac{1}{1+\mu},
  \end{equation*}
  it follows that the third item $I_{3}$ in \eqref{eq:gerror} satisfies
  \begin{equation*}
    \begin{split}
      \|I_{3}\|&\leq c\sum_{p=1}^{+\infty}\int_{\Gamma_{0,\tau}^{\theta}}|e^{st_n}||\omega(e^{-s\tau})|^{1-\alpha}|s+i2p\pi/\tau|^{-\mu-2}|\mathrm{d}s|
                  \\
                 &\leq c \tau^{2+\mu}\sum_{p=1}^{+\infty}p^{-\mu-2}\int_{0}^{\frac{\pi}{\tau\sin\theta}}e^{ \rho t_n\cos\theta}\rho^{1-\alpha}\mathrm{d}\rho   \\
                 &\leq c t_n^{\alpha-2}\tau^{2+\mu}.
    \end{split}
  \end{equation*}
  Therefore, the result \eqref{eq:81} can be obtained.
\end{proof}

For the case $u^0(x)\in L^2(\Omega)$ and $f(x,t)\equiv0$, we can also obtain the following error estimate by the approach analogous to the proof of Theorem \ref{le:81}.
\begin{theorem}\label{thm:err1}
  Assume $u^0(x)\in L^2(\Omega)$ and $f(x,t)\equiv0$ in \eqref{eq:tfdesub0}-\eqref{eq:ibconditions}. Let $u_{h}$ and $\tilde{u}_{h}^n$ be the solutions to \eqref{eq:sFEO} and the GLBE scheme \eqref{eq:GLBE}, respectively. Then it holds that
  \begin{equation}
    \|u_{h}(t_{n})-\tilde{u}_{h}^{n}\|\leq c\left(t_{n}^{-1}\tau +t_{n}^{\alpha-2}\tau^{2-\alpha}\right)\|u^{0}\|,\quad 1\le n\le N.
  \end{equation}
  In addition, if $u^0(x)\in L^2(\Omega)$ and $f(x,t)$ satisfies Assumption \ref{af}, then we have
  \begin{equation*}
    \|u_{h}(t_{n})-\tilde{u}_{h}^{n}\|\leq c\left((t_{n}^{\alpha+\mu-1}+t_{n}^{-1}\|u^{0}\|)\tau +t_{n}^{\alpha-2}\tau^{2+\mu}+\|u^{0}\|t_{n}^{\alpha-2}\tau^{2-\alpha}\right),
  \end{equation*}
  where $-1<\mu<0$.
\end{theorem}

\begin{remark}
  The error estimates in Theorems \ref{le:81} and \ref{thm:err1} shows that the GLBE scheme \eqref{eq:GLBE} is of first order as $0<\alpha<1$ and $-1<\mu<0$. In addition, the results also hold for $\mu\ge0$ from the proof of Theorems \ref{le:81}.
\end{remark}

\subsection{FBDF22 scheme}
In this subsection, we continue to investigate an alternative fully discrete scheme based on the second-order BDF in order to improve the order of accuracy in time. Throughout this subsection, the same notations $\tilde{U}_{h}(t)$ and $\tilde{u}_{h}(t)$ are used to denote the solutions of the new scheme discussed as follows.

In analogy to \eqref{eq:7.3}, by introducing $\tilde{F}(t)=\int_{0}^{t}F_{h}(\xi)\mathrm{d}\xi$, we define $\tilde{U}_{h}(t)$ as an approximate solution to the semidiscrete scheme \eqref{eq:daU} that satisfies
\begin{equation}\label{eq:8.3}
  \tau^{-\alpha}\sum_{j=0}^{\infty}w_{j}\tilde{U}_{h}(t-j\tau)=\Delta_{h}\tilde{U}_{h}(t)+D_{\tau}\tilde{F}_{h}(t)+D_{\tau}\frac{t^{2-\alpha}}{\Gamma(3-\alpha)}u_{h}^{0}(x)
\end{equation}
for $t>0$, and prescribe $\tilde{U}_{h}(t)=0$ for $t\leq 0$,
where $\tilde{F}_h(t)$ satisfies
\begin{equation}\label{eq:f2}
  D\tilde{F}_h(t)=F_h(t),~~\tilde{F}_h(0)=0
\end{equation}
with $F_h(t)$ given by \eqref{eq:f1}.
The sequence $\{w_{j},~j\ge 0\}$ in \eqref{eq:8.3} satisfies
$\sum\limits_{j=0}^{\infty}w_{j}\xi^{j}=(\frac{3}{2}-2\xi+\frac{1}{2}\xi^{2})^{\alpha}$, and $D_{\tau}$ denotes the second-order backward difference operator such that
\begin{equation}
  D_{\tau}v(t)=\tau^{-1}\Big(\frac{3}{2}v(t)-2v(t-\tau)+\frac{1}{2}v(t-2\tau)\Big).
\end{equation}
Furthermore, we define $\tilde{u}_{h}(t)$ as
\begin{equation}\label{eq:8.4}
  \tilde{u}_{h}(t):=D_{\tau}\tilde{U}_{h}(t)=\tau^{-1}\Big(\frac{3}{2}
  \tilde{U}_{h}(t)-2\tilde{U}_{h}(t-\tau)+\frac{1}{2}\tilde{U}_{h}(t-2\tau)\Big).
\end{equation}
This indicates $\tilde{u}_{h}(t)=0$ for $t\leq 0$ as well.
Then, by taking $t=t_{n}=n\tau$ with $\tau=T/N$ for $n=1,\cdots, N$ in \eqref{eq:8.3} and \eqref{eq:8.4}, we propose a fully discrete scheme, called FBDF22, in the following form
\begin{equation}\label{eq:FBDF22}
  \begin{split}
    &\tau^{-\alpha}\sum_{j=0}^{n}w_{j}\tilde{U}_{h}^{n-j}=\Delta_{h}\tilde{U}_{h}^{n}+D_{\tau}\tilde{F}_{h}^{n}
    +D_{\tau}\frac{t_{n}^{2-\alpha}}{\Gamma(3-\alpha)}u_{h}^{0}(x),\\
    &\tilde{u}_{h}^{n}=\tau^{-1}\Big(\frac{3}{2}\tilde{U}_{h}^{n}-2\tilde{U}_{h}^{n-1}
    +\frac{1}{2}\tilde{U}_{h}^{n-2}\Big),
  \end{split}
\end{equation}
where $\tilde{u}_{h}^{n}:=\tilde{u}_{h}(t_{n})$, $\tilde{U}_{h}^{n}:=\tilde{U}_{h}(t_{n})$
and $\tilde{F}_{h}^{n}:=\tilde{F}_{h}(t_{n})$ with $\tilde{F}(\cdot)$ satisfying \eqref{eq:f2}.

We first apply the above method to solving the problem in Example \ref{ex:8.1} numerically to illustrate its effectiveness for problems with singular source terms. The discrete scheme for the fODE in Example \ref{ex:8.1} is of the form
\begin{equation}\label{eq:5.8a}
  \tau^{-\alpha}\sum_{j=0}^{n}w_{j}U_{n-j}=\lambda U_{n}+D_{\tau}\tilde{F}(t_{n})+D_{\tau}\frac{t_{n}^{2-\alpha}}{\Gamma(3-\alpha)}u^{0},~~
  u_{n}=\tau^{-1}\Big(\frac{3}{2}U_{n}-2U_{n-1}+\frac{1}{2}U_{n-2}\Big)
\end{equation}
for $1\leq n\leq N$ with $N\tau=T$. Table \ref{ta:4} shows the errors and average rates of convergence with different time step sizes and various $\alpha\in (0,1)$ and $\nu-\alpha\in [-1,0)$, where an improved order of $O(\tau^{2})$ is achieved compared with that of the uncorrected SBD scheme in Table \ref{ta:2b} and the GLBE scheme in Table \ref{ta:3}.

\begin{table}[ht!]
  \centering
  \caption{Errors and convergence rates for Example \ref{ex:8.1} at $T=1$ by the scheme \eqref{eq:5.8a}.}	
  \label{ta:4} 
  \small\setlength{\tabcolsep}{4pt}
  \begin{tabular}{|c|c|ccccc|c|}
    \hline
    $\alpha$ & $\nu$ & $N=$160 & 320 & 640 & 1280 & 2560 & rate \\\hline
    0.1 & -0.1 & 2.7838E-06 & 6.9249E-07 & 1.7276E-07 & 4.3298E-08 & 1.3361E-08 & $\approx$ 1.93(2.0)\\
        & -0.5 & 1.9267E-05 & 4.7876E-06 & 1.1934E-06 & 2.9698E-07 & 7.6388E-08 & $\approx$ 1.99(2.0)\\
        & -0.9 & 4.6794E-05 & 1.1611E-05 & 2.8947E-06 & 7.2358E-07 & 1.9135E-07 & $\approx$ 1.98(2.0)\\\hline
    0.5 & -0.1 & 1.4784E-06 & 3.6535E-07 & 9.0645E-08 & 2.2547E-08 & 5.5332E-09 & $\approx$ 2.02(2.0)\\
        & -0.3 & 8.0490E-06 & 1.9935E-06 & 4.9528E-07 & 1.2328E-07 & 3.0805E-08 & $\approx$ 2.01(2.0)\\
        & -0.5 & 1.8146E-05 & 4.5109E-06 & 1.1244E-06 & 2.8072E-07 & 7.0270E-08 & $\approx$ 2.00(2.0)\\\hline
    0.7 & -0.1 & 1.8151E-07 & 3.5697E-08 & 6.8178E-09 & 1.2529E-09 & 2.1549E-10 & $\approx$ 2.43(2.0)\\
        & -0.2 & 3.1158E-06 & 7.6400E-07 & 1.8785E-07 & 4.6271E-08 & 1.1422E-08 & $\approx$ 2.02(2.0)\\
        & -0.3 & 7.1901E-06 & 1.7901E-06 & 4.4659E-07 & 1.1153E-07 & 2.7887E-08 & $\approx$ 2.00(2.0)\\\hline
  \end{tabular}
\end{table}

Next we devote to the error estimate of the FBDF22 scheme \eqref{eq:FBDF22}.

\begin{theorem}\label{th:err2}
  Assume that $u^0(x)\equiv0$ and $f(x,t)$ in \eqref{eq:tfdesub0} satisfies Assumption \ref{af}. Let $u_{h}(t)$ and $\tilde{u}_{h}^{n}$ be the solutions to \eqref{eq:sFEO} and \eqref{eq:FBDF22}, respectively. Then it holds that
  \begin{equation}
    \|u_{h}(t_{n})-\tilde{u}_{h}^{n}\|\leq c\big(t_{n}^{\alpha+\mu-2}\tau^{2}
    +t_{n}^{\alpha-3}\tau^{3+\mu}\big),\quad -1<\mu<0
  \end{equation}
  for $1\leq n\leq N$.
\end{theorem}
\begin{proof}
  Taking the Laplace transform on \eqref{eq:8.3} and \eqref{eq:8.4} yields
  \begin{equation}\label{eq:cUh}
    \widehat{\tilde{U}_{h}}(s)=\big(\omega_{2}(e^{-s\tau})^{\alpha}-\Delta_{h}\big)^{-1}
    \omega_{2}(e^{-s\tau})\big(\hat{\tilde{F}}_{h}(s)+s^{\alpha-3}u_{h}(0)\big)
  \end{equation}
  and $\widehat{\tilde{u}_{h}}(s)=\omega_{2}(e^{-s\tau})\widehat{\tilde{U}_{h}}(s)$,
  where $\omega_{2}(z)=\tau^{-1}(\frac{3}{2}-2z+\frac{1}{2}z^{2})$, and $\hat{\tilde{F}}_{h}(s)$ denotes the Laplace transform of $\tilde{F}_h(t)$.
  By the definition of $\tilde{F}_{h}$ satisfying \eqref{eq:f2}, it yields
  \begin{equation}\label{eq:tFs}
    \|\hat{\tilde{F}}_{h}(s)\|=|s|^{-2}\|\hat{f_{h}}(s)\|= |s|^{-2}\|P_{h}\hat{f}(s)\|\le |s|^{-2}\|\hat{f}(s)\|\le c|s|^{-\mu-3}.
  \end{equation}
  In analogy to \eqref{eq:gerror} in Theorem \ref{le:81}, we can obtain
  \begin{equation}\label{eq:gerror1}
    \begin{split}
      u_{h}(t_{n})-\tilde{u}_{h}^{n}&=\int_{\Gamma_{\varepsilon}^{\theta}\slash\Gamma_{\varepsilon,\tau}^{\theta}}e^{st_{n}}
      s\hat{U}_{h}(s)\mathrm{d}s+\int_{\Gamma_{\varepsilon,\tau}^{\theta}\cup S_{\varepsilon}}e^{s t_{n}}\left(s\hat{U}_{h}(s)-\omega_{2}(e^{-s\tau})\widehat{\tilde{U}_{h}}(s)\right)\mathrm{d}s\\
      &-\sum_{\substack{p=-\infty \\ p\ne 0}}^{+\infty}\int_{\Gamma_{0,\tau}^{\theta}}e^{st_{n} }\omega_{2}(e^{-s\tau})(\omega_{2}(e^{-s\tau})^{\alpha}-\Delta_{h})^{-1}
                       \omega_{2}(e^{-s\tau})\hat{\tilde{F}}_{h}(s+i2\pi p/\tau)\mathrm{d}s \\
                       &:=II_{1}+II_{2}+II_{3}.
    \end{split}
  \end{equation}
  For $t_{n}\geq \tau$, we get the estimate
  \begin{equation*}
    \|II_{1}\|\leq c\tau^{2}\int_{\frac{\pi}{\tau\sin\theta}}^{+\infty}
    e^{\rho t_{n}\cos\theta}\rho^{-\alpha-\mu+1}\mathrm{d}\rho\leq c t_{n}^{\alpha+\mu-2}\tau^{2}.
  \end{equation*}
  From Lemmas \ref{le:g0} and \ref{le:45}, it follows that
  \begin{equation*}
    \|(s^{\alpha}-\Delta_{h})^{-1}
    -\left(\omega_{2}(e^{-s\tau})^{\alpha}-\Delta_{h}\right)^{-1}\|
    \leq c\tau^{2}|s|^{2}|\omega_{2}(e^{-s\tau})|^{-\alpha}
  \end{equation*}
  for $s$ enclosed by curves $\Gamma_{0,\tau}^{\theta}$, $\Im(s)=\pm \pi/\tau$ and $\Gamma$. Furthermore,  we can obtain
  \begin{align*}
    &\|(s^{\alpha}-\Delta_{h})^{-1}s
      -\left(\omega_{2}(e^{-s\tau})^{\alpha}-\Delta_{h}\right)^{-1}\omega_{2}(e^{-s\tau})\|\\
    &\le\|(s^{\alpha}-\Delta_{h})^{-1}(s-\omega_{2}(e^{-s\tau})\|+\|\big((s^{\alpha}-\Delta_{h})^{-1}
      -\left(\omega_{2}(e^{-s\tau})^{\alpha}-\Delta_{h}\right)^{-1}\big)\omega_{2}(e^{-s\tau})\|\\
    &\leq c\tau^{2}\big(|s|^{3-\alpha}+|\omega_{2}(e^{-s\tau})|^{1-\alpha}|s|^2\big).
  \end{align*}
  Then we have
  \begin{equation*}
    \begin{split}
      \|II_{2}\|&\leq\|\int_{\Gamma_{\varepsilon,\tau}^{\theta}\cup S_{\varepsilon}}e^{s t_{n}}\left(s-\omega_{2}(e^{-s\tau})\right)\hat{U}_{h}(s)\mathrm{d}s\|
      +\|\int_{\Gamma_{\varepsilon,\tau}^{\theta}\cup S_{\varepsilon}}e^{s t_{n}}\omega_{2}(e^{-s\tau})\left(\hat{U}_{h}(s)-\widehat{\tilde{U}_{h}}(s)\right)\mathrm{d}s\|\\ &\leq c\tau^{2}\int_{\Gamma_{\varepsilon, \tau}^{\theta}\cup S_{\varepsilon}}|e^{st_{n}}|\left(|s|^{1-\alpha-\mu}+|\omega_{2}(e^{-s\tau})|^{2-\alpha}
      |s|^{-\mu-1}\right)|\mathrm{d}s| \\
          &\leq c t_{n}^{\alpha+\mu-2}\tau^{2}.
    \end{split}
  \end{equation*}
  In addition, it follows from Lemmas \ref{le:46} and \ref{le:fbdf2} that
  \begin{equation*}
    \begin{split}
      \|II_{3}\|&\leq c\sum_{p=1}^{+\infty}\int_{\Gamma_{0,\tau}^{\theta}}|e^{st_n}||\omega_2(e^{-s\tau})|^{2-\alpha}|s+i2p\pi/\tau|^{-\mu-3}|\mathrm{d}s|
                  \\
                 &\leq c \tau^{3+\mu}\sum_{p=1}^{+\infty}p^{-\mu-3}\int_{0}^{\frac{\pi}{\tau\sin\theta}}e^{ \rho t_n\cos\theta}\rho^{2-\alpha}\mathrm{d}\rho   \\
                 &\leq c t_n^{\alpha-3}\tau^{3+\mu}.
    \end{split}
  \end{equation*}
  This completes the proof.
\end{proof}

The error estimate of the FBDF22 scheme \eqref{eq:FBDF22} with $u^0(x)\in L^2(\Omega)$ and $f(x,t)\equiv0$ can also be derived by the similar approach as the proof of Theorem \ref{th:err2} just replacing $\hat{\tilde{F}}_h(s)$ by $s^{\alpha-3}u_h(0)$.

\begin{theorem}\label{thm:err3}
  Assume $u^0(x)\in L^2(\Omega)$ and $f(x,t)\equiv0$ in \eqref{eq:tfdesub0}-\eqref{eq:ibconditions}. Let $u_{h}$ and $\tilde{u}_{h}^n$ be the solutions to \eqref{eq:sFEO} and the FBDF22 scheme \eqref{eq:FBDF22}, respectively. Then we have
  \begin{equation}
    \|u_{h}(t_{n})-\tilde{u}_{h}^{n}\|\leq c\left(t_{n}^{-2}\tau^{2}+ t_{n}^{\alpha-3}\tau^{3-\alpha}\right)\|u^{0}\|,\quad 1\le n\le N.
  \end{equation}
  Furthermore, if $u^0(x)\in L^2(\Omega)$ and $f(x,t)$ satisfies Assumption \ref{af}, then it holds
  \begin{equation*}
    \|u_{h}(t_{n})-\tilde{u}_{h}^{n}\|\leq c\left((t_{n}^{\alpha+\mu-2}+t_{n}^{-2}\|u^{0}\|)\tau^2 +t_{n}^{\alpha-3}\tau^{3+\mu}+\|u^{0}\|t_{n}^{\alpha-3}\tau^{3-\alpha}\right).
  \end{equation*}
\end{theorem}
\begin{remark}
  The error estimates in Theorems \ref{th:err2} and \ref{thm:err3} show that the convergence rate of the FBDF22 scheme \eqref{eq:FBDF22} depends on the parameters $\alpha$ and $\mu$, and is second-order when $-1<\mu<0$. Additionally, the estimate is also valid for $\mu\ge0$.
\end{remark}

\section{Numerical examples}\label{sec:5}
In this section, we report some numerical results to verify the convergence rates of the semidiscrete FEM and fully discrete schemes in Sections \ref{sec:3} and \ref{sec:4}.
\subsection{Numerical results by semidiscrete FEM}
In this subsection, we present two numerical examples by the lumped mass FEM to illustrate the theoretical convergence results in Section \ref{sec:3}, where it shows that the convergence rate of the Galerkin FEM is the same as that of the lumped mass FEM if the mesh is symmetric. Since the exact solutions are unknown, we apply the following formula to calculate the convergence rate
\begin{equation*}
  \text{rate}= \log_4(\|\bar{u}_{2h}(t)-\bar{u}_{h}(t)\|/\|\bar{u}_{h}(t)-\bar{u}_{h/2}(t)\|).
\end{equation*}
\begin{example}\label{exm:s1}
  Let $T=1$ and $\Omega=(0,1)$. Consider the one dimensional problem \eqref{eq:tfdesub0}-\eqref{eq:ibconditions} with $u^0(x)\equiv0$ and $f(x,t)=(1+t^{\mu})x^{-\frac{1}{4}}$, where $-1<\mu<0$.
\end{example}

To generate the finite element discretization, the interval $\Omega=(0,1)$ is equally divided into $M$ subintervals with a mesh size $h=1/M$. As mentioned in \cite{JinLZ:2013}, the eigenvalues and eigenfunctions $(\bar{\lambda}_k^h,\bar{\varphi}_k^h(x))_{k=1}^{M-1}$ of the corresponding one dimensional discrete Laplacian $-\bar{\Delta}_h$ defined by \eqref{eq:bDeltah} satisfies $(-\bar{\Delta}_h\bar{\varphi}_k^h,\psi)_h=\bar{\lambda}_k^h(\bar{\varphi}_k^h,\psi)_h, ~\forall~\psi\in X_h$, and
\begin{equation*}
  \bar{\lambda}_k^h=\frac{4}{h^2}\sin^2\frac{\pi k}{2M},\quad
  \bar{\varphi}_k^h(x_i)=\sqrt{2}\sin(k\pi x_i),\quad k=1,\cdots,M-1,
\end{equation*}
where $x_i$ is a mesh point. Then, the solution to the lumped mass FEM scheme \eqref{eq:slmFE} with $u^0(x)\equiv0$ and $f(x,t)=t^{\mu}g(x)$ can be represented as
\begin{equation*}
  \bar{u}_h(t)=\sum_{k=1}^{M-1}(\bar{P}_h g,\bar{\varphi}_k^h)_h\bar{\varphi}_k^h
  \int_0^t(t-s)^{\alpha-1}
  E_{\alpha,\alpha}(-\bar{\lambda}_k^h(t-s)^{\alpha})s^{\mu}\mathrm{d}s,
\end{equation*}
where $E_{\alpha,\beta}(x)$ denotes the Mittag-Leffler function, which can be evaluated by the algorithm developed by \cite{Garrappa:2015}.

\begin{table}[ht]
\centering
\caption{Errors and convergence rates by the scheme \eqref{eq:slmFE} for Example \ref{exm:s1}.}\label{tab:s1}
\setlength{\tabcolsep}{2.5pt}
\begin{tabular}{|c|c|ccccc|c|}
  \hline
  $\alpha$ & $\mu$ & $h=1/16$ & 1/32 & 1/64 & 1/128 & 1/256  & rate  \\
  \hline
  0.1 & -0.1 &  2.84935E-03 & 7.12046E-04 & 1.76860E-04 & 4.37131E-05 &  1.07539E-05  & 2.01  \\
       & -0.5 &  2.85305E-03 & 7.12874E-04 & 1.77054E-04 & 4.37589E-05 &  1.07647E-05  & 2.01  \\
       & -0.9 &  2.87807E-03 & 7.18463E-04 & 1.78363E-04 & 4.40686E-05 &  1.08378E-05  & 2.01  \\
    \hline
  0.5 & -0.1 &  2.87595E-03 & 7.18032E-04 & 1.78268E-04 & 4.40474E-05 &  1.08331E-05  & 2.01  \\
       & -0.5 &  2.89008E-03 & 7.21148E-04 & 1.78992E-04 & 4.42177E-05 &  1.08730E-05  & 2.01  \\
       & -0.9 &  2.96267E-03 & 7.37084E-04 & 1.82691E-04 & 4.50867E-05 &  1.10767E-05  & 2.02  \\
    \hline
  0.9 & -0.1 &  2.90264E-03 & 7.23869E-04 & 1.79618E-04 & 4.43634E-05 &  1.09068E-05  & 2.01  \\
       & -0.5 &  2.90927E-03 & 7.25113E-04 & 1.79876E-04 & 4.44175E-05 &  1.09178E-05  & 2.01  \\
       & -0.9 &  2.89900E-03 & 7.21658E-04 & 1.78908E-04 & 4.41565E-05 &  1.08482E-05  & 2.02  \\
    \hline
\end{tabular}
\end{table}

Table \ref{tab:s1} presents convergence rates by the lumped mass FEM scheme \eqref{eq:slmFE} for Example \ref{exm:s1}. Second-order accuracy is observed, which is consistent with the theoretical estimate \eqref{eq:errslmFEM1} in Theorem \ref{thm:errslmFE}.

\begin{example}\label{exm:s2}
  Let $T=1$ and $\Omega=(0,1)^2$. Consider the two dimensional problem \eqref{eq:tfdesub0}-\eqref{eq:ibconditions} with $u^0(x)\equiv0$ and $f(x,t)=(1+t^{\mu})\chi_{[\frac{1}{4},\frac{3}{4}]\times[\frac{1}{4},\frac{3}{4}]}(x)$, where $-1<\mu<0$ and $\chi_{[\frac{1}{4},\frac{3}{4}]\times[\frac{1}{4},\frac{3}{4}]}(x)$ is the indicator function over $[\frac{1}{4},\frac{3}{4}]\times[\frac{1}{4},\frac{3}{4}]$.
\end{example}

We partition the domain $\Omega=(0,1)^2$ by a uniform symmetric triangulation mesh, where the boundary of $\Omega$ is equally divided into $M$ subintervals with a size $h=1/M$. Then the convergence rates of the Galerkin and lumped mass FEMs are the same. We know from \cite{JinLPZ:2015} that the eigenpairs $(\bar{\lambda}_{n,m}^h,\bar{\varphi}_{n,m}^h(x))_{n,m=1}^{M-1}$ of the corresponding two dimensional discrete Laplacian $-\bar{\Delta}_h$ defined by \eqref{eq:bDeltah} satisfies $(-\bar{\Delta}_h\bar{\varphi}_{n,m}^h,\psi)_h=\bar{\lambda}_{n,m}^h(\bar{\varphi}_{n,m}^h,\psi)_h, ~\forall~\psi\in X_h$, and
\begin{equation*}
  \bar{\lambda}_{n,m}^h=\frac{4}{h^2}\big(\sin^2\frac{n\pi}{2M}+\sin^2\frac{m\pi}{2M}\big),\quad
  \bar{\varphi}_{n,m}^h(x_i,y_k)=2\sin(n\pi x_i)\sin(m\pi y_k)
\end{equation*}
for $n,m=1,\cdots,M-1$, where $(x_i,y_k)$ is a mesh point.
In addition, the approximate solution by the lumped mass FEM scheme \eqref{eq:slmFE} in two dimensional case with $u^0(x)\equiv0$ and $f(x,t)=t^{\mu}g(x)$ can be obtained by
\begin{equation*}
  \bar{u}_h(t)=\sum_{n,m=1}^{M-1}(\bar{P}_hg,\bar{\varphi}_{n,m}^h)_h\bar{\varphi}_{n,m}^h
  \int_0^t(t-s)^{\alpha-1}
  E_{\alpha,\alpha}(-\bar{\lambda}_{n,m}^h(t-s)^{\alpha})s^{\mu}\mathrm{d}s.
\end{equation*}

\begin{table}[ht]
\centering
\caption{Errors and convergence rates by the scheme \eqref{eq:slmFE} for Example \ref{exm:s2}.}\label{tab:s2}
\setlength{\tabcolsep}{2.5pt}
\begin{tabular}{|c|c|ccccc|c|}
  \hline
  $\alpha$ & $\mu$ & $h=1/16$ & 1/32 & 1/64 & 1/128 & 1/256  & rate  \\
  \hline
  0.1 & -0.1 &  1.49059E-03 & 3.85242E-04 & 9.72914E-05 & 2.43964E-05 &  6.10447E-06  & 1.98  \\
       & -0.5 &  1.49322E-03 & 3.85890E-04 & 9.74526E-05 & 2.44366E-05 &  6.11452E-06  & 1.98  \\
       & -0.9 &  1.51096E-03 & 3.90269E-04 & 9.85416E-05 & 2.47084E-05 &  6.18242E-06  & 1.98  \\
    \hline
  0.5 & -0.1 &  1.50830E-03 & 3.89613E-04 & 9.83783E-05 & 2.46676E-05 &  6.17224E-06  & 1.98  \\
       & -0.5 &  1.51937E-03 & 3.92345E-04 & 9.90577E-05 & 2.48372E-05 &  6.21461E-06  & 1.98  \\
       & -0.9 &  1.57967E-03 & 4.07238E-04 & 1.02763E-04 & 2.57618E-05 &  6.44563E-06  & 1.98  \\
    \hline
  0.9 & -0.1 &  1.53039E-03 & 3.95065E-04 & 9.97343E-05 & 2.50060E-05 &  6.25679E-06  & 1.98  \\
       & -0.5 &  1.54219E-03 & 3.97978E-04 & 1.00459E-04 & 2.51868E-05 &  6.30197E-06  & 1.98  \\
       & -0.9 &  1.57305E-03 & 4.05591E-04 & 1.02352E-04 & 2.56593E-05 &  6.42003E-06  & 1.98  \\
    \hline
\end{tabular}
\end{table}

In Table \ref{tab:s2}, the convergence rates obtained by the lumped mass FEM scheme \eqref{eq:slmFE} for Example \ref{exm:s2} are shown. It shows the second-order accuracy of the semi-discrete scheme \eqref{eq:slmFE} with symmetric finite element mesh as predicted in the estimate \eqref{eq:errslmFEM1} in Theorem \ref{thm:errslmFE}.

\subsection{Numerical results by fully discrete schemes}
In this subsection, two numerical examples are presented to verify the theoretical results of two fully discrete schemes in Section \ref{sec:4}.
The numerical results are obtained by the GLBE scheme \eqref{eq:GLBE} and the FBDF22 scheme \eqref{eq:FBDF22}. In the following numerical examples, the exact solutions are unknown, then a reference solution obtained with very small time step size is utilized to evaluate the error $e_h^{\tau}:=\|\tilde{u}_{h,\tau}^N-u(T)\|_{L^2(\Omega)}$, where $\tilde{u}_{h,\tau}^N$ represents the numerical solutions at time $T$ by the fully discrete schemes with the time step size $\tau$ and spatial mesh size $h$.
Then the convergence orders of the two schemes can be verified by the formula
$\log_2(|e_{h}^{2\tau}|/|e_{h}^{\tau}|)$.

\begin{example}\label{exm:f1}
  Let $T=1$ and $\Omega=(0,1)$. Consider the one dimensional problem \eqref{eq:tfdesub0}-\eqref{eq:ibconditions} with the following data:
  \begin{itemize}
    \item[(a)] $u^0(x)\equiv0$ and $f(x,t)=t^{\mu}x^{-\frac{1}{4}}$ with $-1<\mu<0$;
    \item[(b)] $u^0(x)=\chi_{[\frac{1}{4},\frac{3}{4}]}(x)$ and $f(x,t)\equiv0$, where $\chi_{[\frac{1}{4},\frac{3}{4}]}(x)$ is the indicator function over $[\frac{1}{4},\frac{3}{4}]$.
  \end{itemize}
\end{example}

\begin{table}[htb!]
  \centering
  \caption{Errors and convergence rates for case (a) of Example \ref{exm:f1}.}\label{tab:f1a}
  \setlength{\tabcolsep}{4.5pt}\small
  \begin{tabular}{|c|c|c|cccc|c|}
    \cline{1-8}
    Method & $\alpha$ & $\mu$ & $\tau=$1/40 & 1/80 & 1/160 & 1/320 & rate  \\
    \cline{1-8}
    GLBE   & 0.1 & -0.1 & 1.1513E-04 & 5.7668E-05 & 2.8683E-05 & 1.4133E-05 & 1.01 (1.00)  \\
           & 0.1 & -0.5 & 6.2900E-04 & 3.1575E-04 & 1.5719E-04 & 7.7483E-05 & 1.01 (1.00)  \\
           & 0.1 & -0.9 & 1.2347E-03 & 6.2138E-04 & 3.0966E-04 & 1.5269E-04 & 1.01 (1.00)  \\
    \cline{2-8}
           & 0.5 & -0.1 & 7.6705E-05 & 3.8558E-05 & 1.9212E-05 & 9.4743E-06 & 1.01 (1.00)  \\
           & 0.5 & -0.5 & 6.7565E-04 & 3.3928E-04 & 1.6893E-04 & 8.3269E-05 & 1.01 (1.00)  \\
           & 0.5 & -0.9 & 1.9317E-03 & 9.7039E-04 & 4.8314E-04 & 2.3814E-04 & 1.01 (1.00)  \\
    \cline{2-8}
           & 0.9 & -0.1 & 1.1910E-04 & 6.0162E-05 & 3.0042E-05 & 1.4830E-05 & 1.00 (1.00)  \\
           & 0.9 & -0.5 & 9.4222E-04 & 4.7128E-04 & 2.3418E-04 & 1.1531E-04 & 1.01 (1.00)  \\
           & 0.9 & -0.9 & 2.9369E-03 & 1.4524E-03 & 7.1750E-04 & 3.5228E-04 & 1.02 (1.00)  \\
    \cline{1-8}
    FBDF22 & 0.1 & -0.1 & 4.4149E-06 & 1.0789E-06 & 2.6352E-07 & 6.1982E-08 & 2.05 (2.00)  \\
           & 0.1 & -0.5 & 3.3506E-05 & 8.1754E-06 & 2.0116E-06 & 4.9145E-07 & 2.03 (2.00)  \\
           & 0.1 & -0.9 & 8.5065E-05 & 2.0738E-05 & 5.1534E-06 & 1.3195E-06 & 2.00 (2.00)  \\
    \cline{2-8}
           & 0.5 & -0.1 & 2.4546E-06 & 6.0433E-07 & 1.4965E-07 & 3.6935E-08 & 2.02 (2.00)  \\
           & 0.5 & -0.5 & 3.5840E-05 & 8.7531E-06 & 2.1618E-06 & 5.3632E-07 & 2.02 (2.00)  \\
           & 0.5 & -0.9 & 1.3690E-04 & 3.3254E-05 & 8.1879E-06 & 2.0308E-06 & 2.02 (2.00)  \\
    \cline{2-8}
           & 0.9 & -0.1 & 3.7018E-06 & 9.1485E-07 & 2.2762E-07 & 5.6738E-08 & 2.01 (2.00)  \\
           & 0.9 & -0.5 & 5.3813E-05 & 1.3104E-05 & 3.2321E-06 & 8.0203E-07 & 2.02 (2.00)  \\
           & 0.9 & -0.9 & 2.3420E-04 & 5.6570E-05 & 1.3878E-05 & 3.4321E-06 & 2.03 (2.00)  \\
    \cline{1-8}
  \end{tabular}
\end{table}

\begin{table}[htb!]
\centering
\caption{Errors and convergence rates for case (b) of Example \ref{exm:f1}.}\label{tab:f1b}
\setlength{\tabcolsep}{6.0pt}\small
  \begin{tabular}{|c|c|cccc|c|}
    \cline{1-7}
    Method & $\alpha$ & $\tau=$1/40 & 1/80 & 1/160 & 1/320 & rate  \\
    \cline{1-7}
    GLBE   & 0.1 & 6.1206E-05 & 3.0657E-05 & 1.5248E-05 & 7.5128E-06 & 1.01 (1.00)  \\
           & 0.5 & 2.1663E-04 & 1.0878E-04 & 5.4162E-05 & 2.6698E-05 & 1.01 (1.00)  \\
           & 0.9 & 1.7560E-04 & 8.6843E-05 & 4.2901E-05 & 2.1063E-05 & 1.02 (1.00)  \\
    \cline{1-7}
    FBDF22 & 0.1 & 2.3469E-06 & 5.7351E-07 & 1.4007E-07 & 3.2928E-08 & 2.05 (2.00)  \\
           & 0.5 & 1.1491E-05 & 2.8066E-06 & 6.9333E-07 & 1.7217E-07 & 2.02 (2.00)  \\
           & 0.9 & 1.4004E-05 & 3.3826E-06 & 8.2982E-07 & 2.0523E-07 & 2.03 (2.00)  \\
    \cline{1-7}
\end{tabular}
\end{table}

The spatial interval $\Omega=(0,1)$ in Example \ref{exm:f1} is equally divided into subintervals with a mesh size $h=1/128$ for the finite element discretization.
The reference solution is obtained with a time step size $\tau=1/(10\times2^{10})$.
In Table \ref{tab:f1a}, the errors and convergence rates of the GLBE and FBDF22 schemes for case (a) of Example \ref{exm:f1} are presented with $\alpha=0.1,0.5,0.9$ and $\mu=-0.1,-0.5,-0.9$. From the results, we observe that the proposed GLBE scheme converges with rate $O(\tau)$ and the FBDF22 scheme exhibits convergence rate of $O(\tau^{2})$ for $\mu\in(-1,0)$. These are consistent with our theoretical analyses and show the effectiveness of the schemes for solving the problem \eqref{eq:tfdesub0}-\eqref{eq:ibconditions} with the singular source term $f(x,t)$.
In Table \ref{tab:f1b}, we list the errors and convergence rates of the GLBE and FBDF22 schemes for case (b) of Example \ref{exm:f1} with $\alpha=0.1,0.5,0.9$, which verify the theoretical results for the two schemes as well.
\begin{example}\label{exm:f2}
  Let $T=1$ and $\Omega=(0,1)^2$. Consider the two dimensional problem \eqref{eq:tfdesub0}-\eqref{eq:ibconditions} with the following data:
  \begin{itemize}
    \item[(a)] $u^0(x)\equiv0$ and $f(x,t)=(1+t^{\mu})\chi_{[\frac{1}{4},\frac{3}{4}]\times[\frac{1}{4},\frac{3}{4}]}(x)$, where $-1<\mu<0$ and $\chi_{[\frac{1}{4},\frac{3}{4}]\times[\frac{1}{4},\frac{3}{4}]}(x)$ is the indicator function over $[\frac{1}{4},\frac{3}{4}]\times[\frac{1}{4},\frac{3}{4}]$.
    \item[(b)] $u^0(x)=\chi_{[\frac{1}{4},\frac{3}{4}]\times[\frac{1}{4},\frac{3}{4}]}(x)$ and $f(x,t)\equiv0$.
  \end{itemize}
\end{example}

\begin{table}[htb!]
\centering
\caption{Errors and convergence rates for case (a) of Example \ref{exm:f2}.}\label{tab:f2a}
  \setlength{\tabcolsep}{4.5pt}\small
  \begin{tabular}{|c|c|c|cccc|c|}
    \cline{1-8}
    Method & $\alpha$ & $\mu$ & $\tau=$1/80 & 1/160 & 1/320 & 1/640 & rate  \\    \cline{1-8}
    GLBE   & 0.2 & -0.2 & 1.9249E-06 & 9.5759E-07 & 4.7187E-07 & 2.2852E-07 & 1.02 (1.00)  \\
           & 0.2 & -0.5 & 4.8396E-06 & 2.4092E-06 & 1.1875E-06 & 5.7515E-07 & 1.02 (1.00)  \\
           & 0.2 & -0.8 & 7.7741E-06 & 3.8729E-06 & 1.9096E-06 & 9.2496E-07 & 1.02 (1.00)  \\
    \cline{2-8}
           & 0.5 & -0.2 & 1.9130E-06 & 9.5166E-07 & 4.6894E-07 & 2.2709E-07 & 1.02 (1.00)  \\
           & 0.5 & -0.5 & 4.8520E-06 & 2.4154E-06 & 1.1905E-06 & 5.7658E-07 & 1.02 (1.00)  \\
           & 0.5 & -0.8 & 7.8591E-06 & 3.9151E-06 & 1.9303E-06 & 9.3499E-07 & 1.02 (1.00)  \\
    \cline{2-8}
           & 0.8 & -0.2 & 1.9364E-06 & 9.6328E-07 & 4.7463E-07 & 2.2982E-07 & 1.02 (1.00)  \\
           & 0.8 & -0.5 & 4.9017E-06 & 2.4400E-06 & 1.2026E-06 & 5.8237E-07 & 1.02 (1.00)  \\
           & 0.8 & -0.8 & 7.9375E-06 & 3.9538E-06 & 1.9493E-06 & 9.4410E-07 & 1.02 (1.00)  \\
    \cline{1-8}
    FBDF22 & 0.2 & -0.2 & 3.9656E-08 & 9.6940E-09 & 2.2927E-09 & 4.5475E-10 & 2.15 (2.00)  \\
           & 0.2 & -0.5 & 1.2548E-07 & 3.0819E-08 & 7.4747E-09 & 1.6791E-09 & 2.07 (2.00)  \\
           & 0.2 & -0.8 & 2.4331E-07 & 5.9928E-08 & 1.4778E-08 & 3.5753E-09 & 2.03 (2.00)  \\
    \cline{2-8}
           & 0.5 & -0.2 & 3.9386E-08 & 9.6692E-09 & 2.3283E-09 & 5.0521E-10 & 2.09 (2.00)  \\
           & 0.5 & -0.5 & 1.2594E-07 & 3.1005E-08 & 7.5927E-09 & 1.7802E-09 & 2.05 (2.00)  \\
           & 0.5 & -0.8 & 2.4631E-07 & 6.0676E-08 & 1.4974E-08 & 3.6349E-09 & 2.03 (2.00)  \\
    \cline{2-8}
           & 0.8 & -0.2 & 3.9989E-08 & 9.8822E-09 & 2.4450E-09 & 5.9732E-10 & 2.02 (2.00)  \\
           & 0.8 & -0.5 & 1.2752E-07 & 3.1490E-08 & 7.8090E-09 & 1.9293E-09 & 2.02 (2.00)  \\
           & 0.8 & -0.8 & 2.4921E-07 & 6.1460E-08 & 1.5239E-08 & 3.7724E-09 & 2.02 (2.00)  \\
    \cline{1-8}
  \end{tabular}
\end{table}

\begin{table}[htb!]
\centering
\caption{Errors and convergence rates for case (b) of Example \ref{exm:f2}.}\label{tab:f2b}
\setlength{\tabcolsep}{6.0pt}\small
  \begin{tabular}{|c|c|cccc|c|}
    \cline{1-7}
    Method & $\alpha$ & $\tau=$1/80 & 1/160 & 1/320 & 1/640 & rate  \\
    \cline{1-7}
    GLBE   & 0.2 & 1.6614E-06 & 8.2643E-07 & 4.0719E-07 & 1.9715E-07 & 1.03 (1.00)  \\
           & 0.5 & 2.7484E-06 & 1.3682E-06 & 6.7433E-07 & 3.2655E-07 & 1.02 (1.00)  \\
           & 0.8 & 1.7322E-06 & 8.6283E-07 & 4.2539E-07 & 2.0603E-07 & 1.02 (1.00)  \\
    \cline{1-7}
    FBDF22 & 0.2 & 3.4380E-08 & 8.5296E-09 & 2.1440E-09 & 5.5712E-10 & 1.98 (2.00)  \\
           & 0.5 & 7.1369E-08 & 1.7630E-08 & 4.3770E-09 & 1.0865E-09 & 2.01 (2.00)  \\
           & 0.8 & 5.4360E-08 & 1.3405E-08 & 3.3225E-09 & 8.2119E-10 & 2.02 (2.00)  \\
    \cline{1-7}
  \end{tabular}
\end{table}

For the finite element approximation, the domain $\Omega=(0,1)^2$ in Example \ref{exm:f2} is uniformly partitioned into triangles with the mesh size $h=1/128$.
The reference solution is obtained with a time step size $\tau=1/(10\times2^{10})$.
In Tables \ref{tab:f2a}-\ref{tab:f2b}, the errors and convergence rates of the GLBE and FBDF22 schemes for cases (a) and (b) of Example \ref{exm:f2} are shown, respectively. It reveals that the proposed GLBE and FBDF22 schemes perform effectively and converge numerically by the theoretical rates for the problem \eqref{eq:tfdesub0}-\eqref{eq:ibconditions} with the singular source term $f(x,t)$.

\section{Conclusions}
In this paper, we investigate the numerical discretization of sub-diffusion equations with certain type of singular source terms, for which the existing time-stepping schemes lost their optimal convergence order far below one. We first discuss the well-posedness of solutions to inhomogeneous problems with zero initial value. Furthermore, we construct the spatially semidiscrete schemes using linear FEM and lumped mass FEM. In terms of discretizations in time, two fully discrete schemes for the problem, namely GLBE and FBDF22 schemes are proposed and discussed in details, which have first- and second-order accuracy in time, respectively. In addition, we develop the Laplace transform technique to establish the error estimates both in space and time.


\appendix
\section{Proofs of Theorems \ref{thm:errsFE} and \ref{thm:errslmFE}}
\subsection{Proof of Theorem \ref{thm:errsFE}}\label{sec:app1}
\begin{proof}[Proof of Theorem \ref{thm:errsFE}]
  We obtain from \eqref{eq:tildeu} and \eqref{eq:uh} that
  \begin{equation}
    u(t)-u_h(t)=\frac{1}{2\pi i}\int_{\Gamma_{\varepsilon}^{\theta}\cup S_{\varepsilon}}
    e^{st}\hat{G}_h(s)\hat{f}(x,s)\mathrm{d}s,
  \end{equation}
  where
  \begin{equation}\label{eq:hGs}
    \hat{G}_h(s)=(s^{\alpha}-\Delta)^{-1}-(s^{\alpha}-\Delta_{h})^{-1}P_h.
  \end{equation}
  By the similar estimate in the proof of Theorem \ref{thm:wpr} in \cite{LubichST1996}, it holds that $\|\hat{G}_h(s)\|\le Ch^2.$ Then for $\varepsilon=t^{-1}$, together with the condition $\|\hat{f}(s)\|\le c|s|^{-\mu-1}$ in Assumption \ref{af}, we have
  \begin{equation}\label{eq:errsFEM1}
    \begin{aligned}
      \|u(t)-u_h(t)\|&\le ch^2\int_{\Gamma_{\varepsilon}^{\theta}\cup S_{\varepsilon}}|e^{st}||s|^{-\mu-1}|\mathrm{d}s|\\
      &\le ch^2\left(\int_{\varepsilon}^{\infty}e^{\rho t\cos\theta }\rho^{-\mu-1}\mathrm{d}\rho+\int_{-\theta}^{\theta}e^{\varepsilon t\cos\xi}\varepsilon^{-\mu}\mathrm{d}\xi\right)\\
      &\le ct^{\mu}h^2,
    \end{aligned}
  \end{equation}
  which completes the proof.
\end{proof}

\subsection{Proof of Theorem \ref{thm:errslmFE}}\label{sec:app2}
For the convergence analysis of the lumped mass FEM, the quadrature
error operator $Q_h : X_h\rightarrow X_h$ was introduced in \cite{ChatzipantelidisLT:2012}, which is defined by
\begin{equation}\label{eq:Qh}
  (\nabla Q_h v,\nabla w)=(v,w)_h-(v,w),~~\forall~w\in X_h.
\end{equation}
It was analyzed in \cite{ChatzipantelidisLT:2012} that the quadrature error operator $Q_h$ due to mass lumping satisfies the following estimates.
\begin{lemma}[\cite{ChatzipantelidisLT:2012}]
  Let the operators $\bar{\Delta}_h$ and $Q_h$ defined by \eqref{eq:bDeltah} and \eqref{eq:Qh}, respectively. Then it holds that
  \begin{equation}\label{eq:eQh}
    \|\nabla Q_h\psi\|+h\|\bar{\Delta}_hQ_h\psi\|\le Ch^{p+1}\|\nabla^p\psi\|,\quad\forall~\psi\in X_h,~p=0,1.
  \end{equation}
  Furthermore, if the mesh is symmetric, then it satisfies
  \begin{equation}\label{eq:seQh}
    \|Q_h\psi\|\le ch^2\|\psi\|,\quad\forall~\psi\in X_h.
  \end{equation}
\end{lemma}

With the quadrature error operator defined by \eqref{eq:Qh}, it yields from \eqref{eq:sFE} and \eqref{eq:slmFE} that the error $e_h(t)=u_h(t)-\bar{u}_h(t)$ satisfies
\begin{equation}\label{eq:eslmFE}
  {^C}D^{\alpha}_te_{h}(t)-\bar{\Delta}_{h}e_{h}(t)=-\bar{\Delta}_hQ_h{^C}D^{\alpha}_tu_{h}(t),~~\forall~t>0,\quad e_h(0)=0.
\end{equation}
Taking the Laplace transform on \eqref{eq:eslmFE} implies that
\begin{equation}\label{eq:heh}
  \hat{e}_h(s)=(s^{\alpha}-\bar{\Delta}_h)^{-1}\bar{\Delta}_hQ_h\big(s^{\alpha-1}P_hu^0-s^{\alpha}\hat{u}_h(s)\big).
\end{equation}
Since $\Delta_h$ satisfies the resolvent estimate $\|\big(s-\Delta_h\big)^{-1}\|\le M|s|^{-1}$, it is derived from \eqref{eq:LTsFEO} that
\begin{equation}\label{eq:huhs}
  \begin{aligned}
    \|\hat{u}_h(s)\|&=\big\|(s^{\alpha}-\Delta_h)^{-1}\big(s^{\alpha-1}P_hu^0+P_h\hat{f}(s)\big)\big\|\\
    &\le|s|^{-1}\|u^0\|+|s|^{-\alpha-\mu-1}.
  \end{aligned}
\end{equation}
In addition, the operator $\bar{\Delta}_h$ defined by \eqref{eq:bDeltah} also satisfies the resolvent estimate, then it follows from \eqref{eq:eQh} that
\begin{equation}\label{eq:teeh}
  \begin{aligned}
    &\|s^{\alpha-1}(s^{\alpha}-\bar{\Delta}_h)^{-1}\bar{\Delta}_hQ_hP_hu^0\|\le|s|^{-1}\|u^0\|,\\
    &\|s^{\alpha}(s^{\alpha}-\bar{\Delta}_h)^{-1}\bar{\Delta}_hQ_h\hat{u}_h(s)\|
    \le\|\hat{u}_h(s)\|.
  \end{aligned}
\end{equation}
Therefore, by \eqref{eq:huhs}, \eqref{eq:teeh} and the Cauchy's theorem, the inverse Laplace transform on \eqref{eq:heh} implies that the error $e_{h}(t)$ for $t>0$ can be represented by an integral over $\Gamma_{\varepsilon}^{\theta}\cup S_{\varepsilon}$ as follows
\begin{equation}\label{eq:eh}
  e_{h}(t)=\frac{1}{2\pi i}\int_{\Gamma_{\varepsilon}^{\theta}\cup S_{\varepsilon}}
  e^{st}(s^{\alpha}-\bar{\Delta}_{h})^{-1}\bar{\Delta}_hQ_h\big(s^{\alpha-1}P_hu^0-s^{\alpha}\hat{u}_h(s)\big)\mathrm{d}s.
\end{equation}

Now it is ready to establish the error estimate for the lumped mass finite element scheme \eqref{eq:slmFE}. The error is splitted into $u(t)-\bar{u}_h(t)=u(t)-u_h(t)+e_h(t)$ with $u_h(t)$ being the solution of the standard Galerkin finite element scheme in \eqref{eq:sFE}. Since the error $\|u(t)-u_h(t)\|$ is estimated in Theorems \ref{thm:errsFE} and \ref{thm:errsFE1}, then we next focus on the estimate of $\|e_h(t)\|$.

\begin{proof}[Proof of Theorem \ref{thm:errslmFE}]
  As $(s^{\alpha}-\bar{\Delta}_{h})^{-1}\bar{\Delta}_h=s^{\alpha}(s^{\alpha}-\bar{\Delta}_{h})^{-1}-I$ with $I$ being the identity operator, it follows from the resolvent estimate of $\bar{\Delta}_{h}$ that $\|(s^{\alpha}-\bar{\Delta}_{h})^{-1}\bar{\Delta}_h\|\le M+1$. Then we derive from \eqref{eq:eh} and \eqref{eq:huhs} with $u^0(x)\equiv0$ that
  \begin{equation}\label{eq:errslmFEM2}
    \|e_h(t)\|\le c\int_{\Gamma_{\varepsilon}^{\theta}\cup S_{\varepsilon}}|e^{st}||s|^{\alpha}\|Q_h\hat{u}_h(s)\||\mathrm{d}s|.
  \end{equation}
  By the trivial inequality $\|\psi\|\le c\|\nabla\psi\|$ for $\psi\in X_h$ and the estimate \eqref{eq:eQh}, it yields
  $$\|Q_h\hat{u}_h(s)\|\le\|\nabla Q_h\hat{u}_h(s)\|\le ch\|\hat{u}_h(s)\|.$$
  Together with \eqref{eq:huhs}, we obtain the estimate \eqref{eq:errslmFEM0} by \eqref{eq:errsFEM0} and the following argument
  \begin{equation*}
    \begin{aligned}
      \|e_h(t)\|&\le ch\int_{\Gamma_{\varepsilon}^{\theta}\cup S_{\varepsilon}}|e^{st}||s|^{-\mu-1}|\mathrm{d}s|\\
      &\le ch\left(\int_{\varepsilon}^{\infty}e^{\rho\cos\theta t}\rho^{-\mu-1}\mathrm{d}\rho+\int_{-\theta}^{\theta}e^{\varepsilon t\cos\xi}\varepsilon^{-\mu}\mathrm{d}\xi\right)\\
      &\le ct^{\mu}h.
    \end{aligned}
  \end{equation*}

  If the quadrature error operator $Q_h$ satisfies \eqref{eq:seQh}, then it follows the estimate \eqref{eq:errslmFEM1} from \eqref{eq:errsFEM0} and
  \begin{equation*}
    \begin{aligned}
      \|e_h(t)\|&\le c\int_{\Gamma_{\varepsilon}^{\theta}\cup S_{\varepsilon}}|e^{st}||s|^{\alpha}\|\hat{u}_h(s)\||\mathrm{d}s|
      \le ch^2\int_{\Gamma_{\varepsilon}^{\theta}\cup S_{\varepsilon}}|e^{st}||s|^{-\mu-1}|\mathrm{d}s|\\
      &\le ct^{\mu}h^2.
    \end{aligned}
  \end{equation*}
  This completes the proof.
\end{proof}

\section{Four lemmas}
We provide four preliminary lemmas for the error analysis of the GLBE scheme \eqref{eq:GLBE} and the FBDF22 scheme \eqref{eq:FBDF22}.

\begin{lemma}\label{le:g0}
  If $z\in\Sigma_{\pi/2}$, then $(1-e^{-z})^{\alpha}\in\Sigma_{\alpha\pi/2}$. Otherwise if $z\in\Sigma_{\theta}\setminus\Sigma_{\pi/2}$ and its imaginary part satisfying $|\Im(z)|\le\pi$ for $\theta\in(\pi/2,\pi)$, then $(1-e^{-z})^{\alpha}\in\Sigma_{\alpha\theta}$.
\end{lemma}
\begin{proof}
  If $z=x+iy\in\Sigma_{\pi/2}$, then $x>0$. This yields that the real part of $(1-e^{-z})$ satisfies
  \begin{equation*}
    \Re (1-e^{-z})=1-e^{-x}\cos y>0
  \end{equation*}
  for all $y\in\mathrm{R}$. Then we obtain $(1-e^{-z})\in\Sigma_{\pi/2}$ and $(1-e^{-z})^{\alpha}\in\Sigma_{\alpha\pi/2}$.

  Otherwise if $z\in\Sigma_{\theta}\setminus\Sigma_{\pi/2}$, then $x\leq0$ and $x\tan\theta \leq| y|\leq \pi$. It suffices to consider the case $\Re (1-e^{-z})<0$. We define
  \begin{equation*}
    f(x,y):=\Big|\frac{\Im(1-e^{-z})}{\Re(1-e^{-z})}\Big|=\frac{|\sin y|}{\cos y -e^{x}}.
  \end{equation*}
  From
  \begin{equation*}
    \frac{\partial f(x,y)}{\partial y}
    =\frac{1-e^{x}\cos y}{\left(\cos y -e^{x}\right)^{2}}\geq 0\quad\text{for}\quad x\tan\theta\leq y\leq \pi
  \end{equation*}
  and
  \begin{equation*}
    \frac{\partial f(x,y)}{\partial y}=\frac{e^{x}\cos y-1}{\left(\cos y-e^{x}\right)^{2}}\leq 0
    \quad \text{for}\quad -\pi\leq y\leq -x\tan\theta<0,
  \end{equation*}
  it follows that $f(x,y)\geq f(x,x\tan\theta)$ for $x\leq0$. Taking the derivative of $\tilde{f}(x,{\theta}):=f(x,x\tan\theta)$ with respect to $x$ arrives at
  \begin{equation*}
    \frac{\partial \tilde{f}}{\partial x}=\frac{\tan\theta+e^{x}\sin(x\tan\theta)-e^{x}\cos(x\tan\theta)\tan\theta}{\left(\cos (x\tan\theta)-e^{x}\right)^{2}}=:\frac{g(x,\theta)}{\left(\cos (x\tan\theta)-e^{x}\right)^{2}}.
  \end{equation*}
  From $\partial g/\partial x=e^{x}\sin(x\tan\theta)(1+\tan^{2}\theta)\geq 0$,
  it follows that $g(x,\theta)\leq g(0,\theta)=0$. This leads to $\tilde{f}(x,{\theta})\geq \tilde{f}(0,{\theta})=-\tan\theta$ for any $\theta\in (\pi/2,\pi)$.
  Therefore, we obtain $(1-e^{-z})\in\Sigma_{\theta}$ and the desired result.
\end{proof}

\begin{lemma} \label{le:45}
  If $z\in\mathbb{C}$ and $|z|\le r$ for finite $r>0$, then
  \begin{equation}
    |1-e^{-z}|\leq C|z|
  \end{equation}
  and
  \begin{equation}
    |z^{\beta}-(1-e^{-z})^{\beta}|\leq C|z|^{\beta+1}
  \end{equation}
  hold for $0<\beta\leq 1$, where $C$ denotes a generic constant dependent on the radius $r$.
\end{lemma}

\begin{proof}
  Using Taylor's expansion of $e^{-z}$ at $0$, we derive
  \begin{equation*}
     |1-e^{-z}|=|1-\sum_{n=0}^{\infty}\frac{(-z)^{n}}{\Gamma(n+1)}|\leq e^{|z|}-1\leq \frac{e^{r}-1}{r} |z|
  \end{equation*}
  and similarly
  \begin{equation*}
    |z-(1-e^{-z})|=|z-1+\sum_{n=0}^{\infty}\frac{(-z)^{n}}{\Gamma(n+1)}|
    \leq e^{|z|}-1-|z|\leq\frac{e^{r}-1-r}{r^{2}}|z|^{2}.
  \end{equation*}
  Furthermore, for $0<\beta<1$, by the result in \cite{JinLZ2017,LubichST1996}, we have
  \begin{equation*}
    |z^{\beta}-(1-e^{-z})^{\beta}|\leq \max\{|z|^{\beta-1}, |1-e^{-z}|^{\beta-1}\} |z-(1-e^{-z})|,
  \end{equation*}
  which completes the proof.
\end{proof}

\begin{lemma} \label{le:46}
  Let $|z|\le r$ for finite $r$. Then it holds that
  \begin{equation}\label{eq:4.2.1}
    \begin{split}
      &|\frac{3}{2}-2e^{-z}+\frac{1}{2}e^{-2z}|\leq C|z|, \\
      &|z^{\beta}-(\frac{3}{2}-2e^{-z}+\frac{1}{2}e^{-2z})^{\beta}|\leq C|z|^{\beta+2}
    \end{split}
  \end{equation}
  for $0<\beta\le 1$, where $C=C(r)$.
\end{lemma}
\begin{proof}
  Using Taylor's expansions of $e^{-z}$ and $e^{-2z}$ at $z=0$ arrives at
  \begin{equation*}
    \begin{split}
      \Big|\frac{3}{2}-2e^{-z}+\frac{1}{2}e^{-2z}\Big|
      &=\Big|-2\sum_{n=1}^{\infty}\frac{(-z)^{n}}{\Gamma(n+1)}
      +\frac{1}{2}\sum_{n=1}^{\infty}\frac{(-2z)^{n}}{\Gamma(n+1)}\Big|    \\
      &\leq 2(e^{|z|}-1)+\frac{1}{2}(e^{2|z|}-1) \\
      &\leq \Big(\frac{1}{2}e^{2r}+2e^{r}-\frac{5}{2}\Big)r^{-1}|z|
    \end{split}
  \end{equation*}
  and
  \begin{equation*}
    \begin{split}
      \Big|z-\frac{3}{2}+2e^{-z}-\frac{1}{2}e^{-2z}\Big|
      &\leq 2\Big(e^{|z|}-\frac{1}{2}|z|^{2}-|z|-1\Big)+\frac{1}{2}\big(e^{2|z|}-2|z|^{2}-2|z|-1\big) \\
      &\leq \Big(\frac{1}{2}e^{2r}+2e^{r}-2r^{2}-3r-\frac{5}{2}\Big)r^{-3}|z|^{3}.
    \end{split}
  \end{equation*}
  Then the second estimate of \eqref{eq:4.2.1} is derived from the approach proposed in \cite{LubichST1996,JinLZ2017}.
\end{proof}
\begin{lemma}\label{le:fbdf2}
  If $z\in\Sigma_{\pi/2}$, then $(\frac{3}{2}-2e^{-z}+\frac{1}{2}e^{-2z})^{\alpha}\in\Sigma_{\alpha\pi/2}$. Otherwise if $z\in\Sigma_{\theta}\setminus\Sigma_{\pi/2}$ for $\theta\in (\pi/2, \tilde{\theta})$ with some $\tilde{\theta}\in(\pi/2,\pi)$ and $|\Im z|\le\pi$, then there corresponds some $\vartheta\in (\pi/2,\pi)$ such that $(\frac{3}{2}-2e^{-z}+\frac{1}{2}e^{-2z})^{\alpha}\in\Sigma_{\alpha\vartheta}$.
\end{lemma}
\begin{proof}
  For $z=x+iy\in\Sigma_{\pi/2}$, it follows that $x>0$ and then
  \begin{equation*}
    \begin{split}
      \Re\Big(\frac{3}{2}-2e^{-z}+\frac{1}{2}e^{-2z}\Big)&=\frac{3}{2}-2e^{-x}\cos y+\frac{1}{2}e^{-2x}\cos(2y)\\
      &=\frac{1}{2}(1-e^{-2x})+(1-e^{-x}\cos y)^{2}>0.
    \end{split}
  \end{equation*}
  This yields $\left(\frac{3}{2}-2e^{-z}+\frac{1}{2}e^{-2z}\right)^{\alpha}\in \Sigma_{\alpha\pi/2}$.

  If $z\in\Sigma_{\theta}\setminus\Sigma_{\pi/2}$ and $|\Im z|\le\pi$, then $x\tan\theta\leq|y|\leq \pi$ and $x\in[\pi/\tan\theta,0]$. It yields
  \begin{equation*}
    \Im(\frac{3}{2}-2e^{-z}+\frac{1}{2}e^{-2z})=e^{-x}\sin y(2-e^{-x}\cos y).
  \end{equation*}
  Set $g(x,y)=2-e^{-x}\cos y$. We find that $\partial g/\partial y\geq 0$ for $y\in [x\tan\theta,\pi]$ and $\partial g/\partial y\leq 0$ for $y\in [-\pi,-x\tan\theta]$. This leads to $g(x,y)\geq g(x,x\tan\theta)$. Then taking the derivative of $g(x,x\tan\theta)$ with respect to $x$, we get
  \begin{equation*}
    g(x,y)\geq 2-e^{-(\theta-\pi/2)/\tan\theta}\sin\theta:=\tilde{g}(\theta)>\tilde{g}(\tilde{\theta})=0
  \end{equation*}
  for $\theta\in(\pi/2, \tilde{\theta})$, where $\tilde{\theta}$ is implicitly determined by $\tilde{g}(\tilde{\theta})=0$ as $\mathrm{d}\tilde{g}/\mathrm{d}\theta<0$ for $\theta\in(\pi/2,\pi)$.

  Next it suffices to consider the case $\theta\in (\pi/2,\tilde{\theta})$ and  $\Re(\frac{3}{2}-2e^{-z}+\frac{1}{2}e^{-2z})<0$. Let
  \begin{equation*}
    f(x,y):=\Big|\frac{\Im(\frac{3}{2}-2e^{-z}+\frac{1}{2}e^{-2z})}{\Re(\frac{3}{2}-2e^{-z}+\frac{1}{2}e^{-2z})}\Big|
    =\frac{e^{-x}|\sin y|(2-e^{-x}\cos y)}{-\frac{3}{2}+2e^{-x}\cos y-\frac{1}{2}e^{-2x}\cos(2y)}.
  \end{equation*}
  For $y\in [0, \pi]$, we obtain
  \begin{equation*}
    \frac{\partial f(x,y)}{\partial x}=\frac{e^x\sin y\left(1-3e^{x}\cos y+3e^{2x}\right)}{\left(\frac{3}{2}e^{2x}-2e^{x}\cos y+\frac{1}{2}\cos(2y)\right)^{2}}\geq 0
  \end{equation*}
  in view of
  \begin{equation*}
    1-3e^{x}\cos y+3e^{2x}\geq 1-3e^{x}+3e^{2x}=3(e^{x}-\frac{1}{2})^{2}+\frac{1}{4}>0.
  \end{equation*}
  Together with $f(x,y)=f(x,-y)$, it holds that
  \begin{equation*}
    f(x,\pm y)\geq f(y/\tan\theta,y):=\tilde{f}(y,\theta)
  \end{equation*}
  for $0\leq y\leq \pi$. Similarly, from the relation $\frac{\partial \tilde{f}(y,\theta)}{\partial \theta} =\frac{\partial f}{\partial x}(y/\tan\theta,y)\frac{\mathrm{d}(y/\tan\theta)}{\mathrm{d}\theta}<0$ together with  $\tilde{f}(0,\theta)=-\tan\theta>0$ and $\Re(3/2-2e^{-z}+1/2e^{-2z})>0$ for $y=\pi$, it follows that $\tilde{f}(y,\theta)>0$ for any $\theta\in (\pi/2,\tilde{\theta})$. Then taking the derivative of $\tilde{f}$ with respect to $y$, we deduce that
  \begin{equation*}
    \tilde{f}(y,\theta)\geq \tilde{f}(y_{\theta},\theta)>0,
  \end{equation*}
  where $y_{\theta}$ satisfies $\frac{\partial\tilde{f}}{\partial y}(y_{\theta},\theta)=0$.
  Thus there corresponds some $\vartheta\in (\pi/2,\pi)$, defined by  $\tilde{f}(y_{\theta},\theta)=-\tan(\vartheta)$ such that $(\frac{3}{2}-2e^{-z}+\frac{1}{2}e^{-2z})\in\Sigma_{\vartheta}$. This completes the proof.
\end{proof}

\bibliographystyle{siam}

\end{document}